\RequirePackage{luatex85}
\documentclass[10pt,a4paper]{article}
\usepackage[utf8]{inputenc}
\usepackage[english]{babel}

\usepackage[T1]{fontenc}      
\usepackage[utf8]{inputenc}
\usepackage{a4}
\usepackage{setspace}
\usepackage{mathrsfs}
\usepackage{epsfig}
\usepackage[numbers]{natbib}

\usepackage{amsmath}
\usepackage{amssymb}
\usepackage{bm}

\usepackage{bbm}
\usepackage{dsfont}
\usepackage{xcolor}
\usepackage{comment}
\usepackage{authblk}

\usepackage{amsfonts}
\usepackage{stmaryrd}
\usepackage{epsf}
\usepackage{amsfonts,bbm}
\usepackage{amsthm}
\usepackage{amscd}
\usepackage{amsopn}

\usepackage{tabularx}
\usepackage{multirow}
\usepackage{pifont}
\usepackage{setspace}
\usepackage{mathrsfs}
\usepackage{natbib}
\usepackage{epsfig}
\usepackage{bm}
\usepackage{bbm}
\usepackage{amssymb}
\usepackage{pifont}
\usepackage[utf8]{inputenc}
\usepackage[T1]{fontenc}
\usepackage{lmodern}
\usepackage{latexsym}
\usepackage{bbm}
\usepackage{amsthm}
\usepackage{relsize}
\usepackage{amsmath}
\usepackage{amssymb}
\usepackage{bigints}
\usepackage{hyperref}
\usepackage[mathcal]{eucal}
\usepackage{multicol}
\usepackage{ulem}
\usepackage{dsfont}
\usepackage{amsfonts,amsmath,amssymb}
\usepackage{fancybox}
\usepackage{amsthm}
\usepackage{amsmath}
\usepackage{amssymb}
\usepackage{mathrsfs}
\usepackage{latexsym}

\usepackage{graphics}  
\usepackage{colortbl}
\usepackage{fancyhdr}
\usepackage{makeidx}
\usepackage{enumitem}
\usepackage{lettrine}
\usepackage{hyperref}
\usepackage{amssymb,wasysym}
\usepackage{stmaryrd}
\usepackage{epsf}
\usepackage{amsfonts,bbm}
\usepackage{amsthm}
\usepackage{amscd}
\usepackage{amsopn}
\usepackage{tabularx}
\usepackage{multirow}
\usepackage{amsfonts}
\newtheorem{theoreme}{Theorem}[section]
\newtheorem{lemme}{Lemma}[section]
\newtheorem{remarque}{Remark}[section]
\newtheorem{proposition}{Proposition}[section]

\newtheorem{definition}{Definition}[section]

\newcommand{\R}{\mathbb{R}}

\newcommand{\N}{\mathbb{N}}

\newcommand{\norm}[1]{\left\Vert#1\right\Vert}
\newcommand{\abs}[1]{\left\vert#1\right\vert}

\newcommand{\bc}{\begin{center}}
\newcommand{\ec}{\end{center}}
\newcommand{\benu}{\begin{enumerate}}
\newcommand{\eenu}{\end{enumerate}}

\newcommand{\real}{\mathbb R}

\newcommand{\fact}[1]{#1\mathpunct{}!\;} 

\title{\bf {Asymptotic properties of $\mathit{AD}(1,n)$ model and its maximum likelihood estimator}}

\author[1]{Mohamed Ben Alaya}
\author[1,2]{Houssem Dahbi\thanks{  \sffamily Corresponding author}}
\author[2]{Hamdi Fathallah}
\affil[1]{Univ Rouen Normandie, CNRS, LMRS UMR 6085, F-76000 Rouen, France.}
\affil[2]{Universit\'e de Sousse, Laboratoire LAMMDA, \'Ecole Sup\'erieure des Sciences et de Technologie de Hammam Sousse, Rue Lamine El Abbessi 4011 Hammam Sousse, Tunisie.}

\DeclareUnicodeCharacter{2212}{-}
\begin{document}
	\maketitle


	\begin{abstract}
		This paper deals with the problem of global parameter estimation of affine diffusions in $\R_+\times\R^n$ denoted by $\mathit{AD}(1,n)$ where $n$ is a positive integer which is a subclass of affine diffusions introduced by Duffie et al in \cite{Duffie}. The $\mathit{AD}(1,n)$ model can be applied to the pricing of bond and stock options, which is illustrated for the Vasicek, Cox-Ingersoll-Ross and Heston models. Our first result is about the classification of $\mathit{AD}(1,n)$ processes according to the subcritical, critical and supercritical cases. Then, we give the stationarity and the ergodicity theorems of this model and we establish asymptotic properties for the maximum likelihood estimator in both  subcritical and a special supercritical cases.
		 
		\end{abstract}
	\makeatletter{\renewcommand*{\@makefnmark}{}
		\footnotetext{E-mail adresses:\\
			\url{mohamed.ben-alaya@univ-rouen.fr}\\
			\url{houssem.dahbi@univ-rouen.fr} and  \url{houssemdahbi@essths.u-sousse.tn}\\
			\url{hamdi.fathallah@essths.u-sousse.tn}
	This research is supported by Laboratory of Excellence MME-DII, Grant no. ANR11-LBX-0023-01
	(\url{http://labex-mme-dii.u-cergy.fr/}).}\makeatother}				

\section{Introduction}
The set of affine processes contains a large class of
important Markov processes such as continuous state branching processes. One of the most important Markov affine models is that of an affine diffusion model which is well characterized on the state space $\R_+^m\times\R^n$ with $m$ and $n$ are non-negative integers (see, e.g., \cite{Duffie}). In this paper, we study the subclass $\mathit{AD}(1,n)$ models where $\mathit{AD}$ stands for affine diffusion, the parameter $1$ refers to the dimension of the first component, $m=1$, which is the Cox-Ingersoll-Ross (CIR) process and $n$ is the dimension of the vector containing the rest of components. Hence, we consider an  $\mathit{AD}(1,n)$ process $Z=(Y,X)^{\mathsf{T}}$ in the $d$-dimensional canonical state space $\mathcal{D}:=[0,\infty)\times (-\infty,\infty)^n$, for
$d=n+1$ and $n$ is a positive integer, strong solution of the following stochastic differential equation (SDE)

\begin{equation}
	\begin{cases}
		\mathrm{d}Y_t=(a-b Y_t) \mathrm{~d}t+\rho_{11}\sqrt{Y_t}\mathrm{~d}B_t^1\\
		\mathrm{d}X_t=(m-\kappa Y_t-\theta X_t) \mathrm{~d}t +\sqrt{Y_t}\tilde{\rho} \mathrm{~d}B_t,
	\end{cases}\quad  t \in \left[ 0, \infty \right), 
	\label{model01}
\end{equation}
where $B_t=(B_t^1,\ldots,B_t^d)^{\mathsf{T}}$, $t\in[0,\infty)$, is a $d$-dimensional independent Brownian vector, $(Y_0,X_0)^{\mathsf{T}}$ is an arbitrary initial value independent of $B$ such that $\mathbb{P}(Y_0\in(0,\infty))=1$, $a\in(0,\infty)$, $m\in(-\infty,\infty)^n$, $b\in(-\infty,\infty)$, $\kappa\in(-\infty,\infty)^n$, $\theta$ is an $n\times n$ real matrix and  $\tilde{\rho}=[\rho_{J1}\;\rho_{JJ}]$, where $J=\lbrace 2,\ldots,d\rbrace$, $\rho_{J1}=(\rho_{i1})_{i\in J}$ and $\rho_{JJ}=(\rho_{ij})_{i,j\in J}$ which are introduced via the block representation of the $d\times d$ positive definite triangular low matrix $\rho=\begin{pmatrix}
	\rho_{11}&\rho_{1J}\\
	\rho_{J1}&\rho_{JJ}
	
\end{pmatrix}$ satisfying $\sum_{j=1}^d {\rho}^2_{ij}=\sigma_i^2$, for all $i\in\lbrace 1,\ldots,d\rbrace$, where $\mathbf{\sigma}=(\sigma_1,\ldots,\sigma_d)^{\mathsf{T}}\in (0,\infty)^d$. Note that, on one hand, according to \cite{Duffie}, $\rho$ is admissible in the sense that $\mathit{AD}(1,n)$ is an affine diffusion. On the other hand, the first component of the model is a CIR process which is also a continuous state branching process with branching mechanism $by+\dfrac{\rho_{11}^2}{2}y^2$ and immigration mechanism $ay$ and conditionally on $\sigma (Y_t, t\in[0,\infty))$ the second component is an $n$-dimensional Vasicek process.

The affine models, in particular the $\mathit{AD}(1,n)$ one, has been employed in finance since the last decades and they have found growing interest due to their computational tractability as well as their capability to capture empirical evidence from financial time series, see, e.g., Duffie et al \cite{Duffie} and  Filipovic and Mayerhofer in \cite{Filipovic}, where they gave more financial applications of affine models and proved existence and uniqueness through stochastic invariance of the canonical state space. 
These applications are illustrated for different known diffusion models as Vasicek, Cox-Ingersoll-Ross and Heston models. Further, stationarity and ergodicity studies for affine diffusion processes found a particular interest by many authors, see, e.g., \cite{Barczy stationarity ergodicty,Boylog1,Friesen,Friesen2,Jin,Jin2,Jin 3} and \cite{Zhang}, where they give sufficient conditions
for the existence of a unique stationary distribution and for the ergodic property. Moreover, several papers have studied the drift parameter estimation in different models, see, e.g., \cite{Barczy,Barczy2,Basak,Alaya,Alaya1,Boylog,Fathallah} and the references therein. It should be noted that statistical inferences of $\mathit{AD}(1,n)$ model were not considered and more generally, inferences of affine multidimensional diffusions are rarely treated. We cite, as well, \cite[Chapter 4]{Alfonsi} and \cite{Dai}, where they introduced some examples and simulations of multidimensional affine diffusions in financial and econometric fields.

Related to our work, from a statistical point of view, in \cite{Alaya} and \cite{Alaya1}, Ben Alaya and Kebaier considered the CIR model, called also the square-root model, which refers to the model $\eqref{model01}$ with $n=0$. they showed asymptotic properties of the maximum likelihood estimators (MLE) in the ergodic case ($b>0$) and non-ergodic cases ($b<0$ and $b=0$).

More later, in \cite{Barczy2}, Barczy et al studied an affine two factor model, which is a submodel of $\eqref{model01}$ with $n=1$, $\kappa=\mathbf{0}$ and $\rho=\mathbf{I}_2$. They showed some asymptotic properties of both maximum likelihood and least squares estimators of drift parameters in the ergodic case based on continuous time observations and proved there strong consistency and asymptotic normality. Slightly more generally, in \cite{Boylog}, Bolyog and Pap studied asymptotic properties of conditional least squares estimators for the drift parameters of two-factor affine diffusions, which corresponds to the case $n=1$ in our model $\eqref{model01}$ by adding a third Wiener process to the SDE related to $X$. 
In the subcritical case ($b>0$ and $\theta>0$), the special critical case ($b = 0$ and $\theta=0$) and the special supercritical case ($\theta < b < 0$), they proved different types of consistency and asymptotic normality of the conditional least square estimators under some additional assumptions on the drift and the diffusion parameters.

From ergodicity point of view, recently in 2020 (see \cite{Jin2}), Jin et al gave a general condition for the stationarity of jump affine processes on a general canonical state space $[0,\infty)^m\times (-\infty,\infty)^n$, for non-negative integers $m$ and $n$. They proved the existence of a limit distribution of the process which can be identified through its characteristic function. Furthermore, for the same model, in \cite{Friesen}, Friesen et al. studied the existence and the regularity of transition densities and proved the exponential ergodicity in the total variation distance under a Hörmander type condition and other assumptions on the parameters.

The aim of this paper is twofold. On one hand, we  will study the stationarity and the ergodicity of the $\mathit{AD}(1,n)$ model. On the other hand, we will apply the obtained results to establish consistency and asymptotic normality of the drift parameter maximum likelihood estimator. 
The paper is organized as follows, in the second section, we introduce at first all the principal notations and tools we need, then we give a classification of the model $\eqref{model01}$ with respect to $b$ and the eigenvalues of $\theta$ according to the asymptotic behavior of $\mathbb{E}(Z_t)$ as $t$ tends to infinity. Essentially, we define the subcritical case when $\mathbb{E}(Z_t)$ converges, the critical case when $\mathbb{E}(Z_t)$ has a polynomial growth and the supercritical case when $\mathbb{E}(Z_t)$ has an exponential growth, for more details see Proposition \ref{classification}. In the third section, using affine properties of our process, we give sufficient conditions for the existence of a unique stationary distribution $Z_{\infty}=(Y_\infty,X_\infty)$ given by the Fourier-Laplace transform

\begin{equation*}
\mathbb{E}\left(e^{-\lambda Y_\infty +i \mu^{\mathsf{T}} X_\infty}\right)=\exp\left(a\displaystyle\int_0^{\infty} \mathcal{K}_s\left(-\lambda,\mu \right)\mathrm{~d}s+i\mu^{\mathsf{T}}\theta^{-1} m\right),
\end{equation*}
where $\lambda$ is a complex number with a positive real part, $\mu$ is a $n$-dimensional real vector and $\mathcal{K}:t\mapsto\mathcal{K}_t(-\lambda,\mu)$ is a complex-valued time function satisfying the Riccati equation $\eqref{EDK}$, for more details see Theorem \ref{Stationarity theorem}. In the fourth section, we establish the exponential ergodicity of $(Y_t,X_t)_{t\in[0,\infty)}$ using the so-called Foster-Lyapunov criteria, see \cite[page 535, Section 6]{Foster-Lyapunov}, namely, we prove that there exists $\delta\in(0,\infty)$ and $B\in(0,\infty)$ such that
\begin{equation*}
	\underset{\vert g\vert\leq V+1}{\sup}\left\vert \mathbb{E}\left(g(Y_t,X_t)\vert (Y_0,X_0)=(y_0,x_0)\right)-\mathbb{E}(g(Y_{\infty},X_{\infty}))\right\vert\leq B(V(y_0,x_0)+1))e^{-\delta t},
\end{equation*}
for all $t\in[0,\infty)$ and for all Borel measurable functions $g:\mathcal{D}\to(-\infty,\infty)$, where $(y_0,x_0)\in\mathcal{D}$ and $V$ is a chosen norm-like function on $\mathcal{D}$, for more details see Theorem \ref{ergodicity theorem}. Note that our result seems to be new an not covered by \cite[Theorem 1.3]{Friesen}. 
The last section is devoted to the study of statistical inferences of the drift parameters maximum likelihood estimators constructed using Lipster and Shiryaev \cite[page 297]{Lipster} associated to the $\mathit{AD}(1,n)$ model. We study its consistency and its asymptotic behavior in both subcritical and a special supercritical  cases (see Theorem \ref{Theorem asymptotic behavior subcritical case} and Theorem \ref{normality theorem supercritical case}).
We close the paper with an appendix, where we recall the strong law of large numbers and the central limit theorem for martingales.

\section{Preliminary}
At first, we start by fixing the used notations. Let $\N,\; \R, \;\mathbb{R}_{+},\; \mathbb{R}_{++}, \; \mathbb{R}_{-}, \; \mathbb{R}_{--}$ and $\mathbb{C}$
denote the sets of non-negative integers, real numbers, non-negative real numbers, positive real numbers, non-positive real numbers, negative real numbers and complex numbers, respectively and let $\mathcal{D}=\R_{+}\times\R^n$, for $n \in \N \setminus \lbrace 0 \rbrace.$ For $x,y\in\R$, we will use the notation $x\wedge y:=\min(x,y)$ and $x\vee y:=\max(x,y)$. 
For all $z\in\mathbb{C}$, \texttt{Re}($z$) denotes the real part of $z$ and \texttt{Im}($z$) denotes the imaginary part of $z$. Let us denote, for $p,q\in\N \setminus \{0\}$,  by $\mathcal{M}_{p,q}$ the set of $p\times q$ real matrices, $\mathcal{M}_p$ the set of $p\times p$ real matrices, $\mathbf{I}_p$ the identity matrix in $\mathcal{M}_{p}$, $\textbf{0}_{p,q}$ the null matrix in $\mathcal{M}_{p,q}$, $\textbf{0}_{p}:=\textbf{0}_{p,1}$ and $\mathbf{1}_p\in\mathcal{M}_{p,1}$ is the 1-vetor of size $p$.
We denote, for all diagonalizable matrix $A\in\mathcal{M}_p$, by $eig(A)$ the column vector containing all the eigenvalues of $A$ and by $\lambda_i(A)$, $\lambda_{\min}(A)$ and $\lambda_{\max}(A)$, the $i^{th}$, the smallest and the largest eigenvalues of $A$, respectively. For $m,n\in\N \setminus \{0\} $, we write $I=\{1,\ldots,m\}$ and $J=\{m+1,m+n\}$ and for all $x\in \R^{m+n}$, we write $\mathbf{x}_I=(x_i)_{i\in I}$ and $\mathbf{x}_J=(x_j)_{j\in J}$. Throughout this paper, we use the notation 
$$A=\begin{bmatrix}
	A_{II}&A_{IJ}\\
	A_{JI}&A_{JJ}
\end{bmatrix}$$
for $A\in\mathcal{M}_{m+n}$ where $A_{II}=(a_{i,j})_{i,j\in{I}}$, $A_{IJ}=(a_{i,j})_{i\in{I},j\in{J}}$, $A_{JI}=(a_{i,j})_{i\in{J},j\in{I}}$ and $A_{JJ}=(a_{i,j})_{i,j\in{J}}$. For all matrices $A_1,\ldots,A_k$, $\text{diag}(A_1,\ldots,A_k)^{\mathsf{T}}$ denotes the block matrix containing the matrices $A_1,\ldots,A_k$ in its diagonal and for all matrix $\textbf{M}$,  let us denote $\norm{\textbf{M}}:=\sqrt{\max\text{eig}(\textbf{M}\textbf{M}^{\mathsf{T}})}$ with $\textbf{M}^{\mathsf{T}}$ denotes the transpose matrix of $\textbf{M}$ and for all vector $\textbf{x}=(x_1,\ldots,x_p)\in\R^p$, we use the notation $\norm{\textbf{x}}:=\displaystyle\sum_{i=1}^p\vert {x}_i\vert$ and $\norm{\mathbf{x}}_2^2:=\displaystyle\sum_{i=1}^p {x}_i^2$.
We denote by $\otimes$ the Kronecker product defined, for all matrices $\mathbf{A}=(a_{ij})_{1\leq i\leq p,\,1\leq j\leq q}$  and $\mathbf{B}$,
$$
\mathbf{A}\otimes\mathbf{B}=\begin{bmatrix}
	a_{11}\mathbf{B}&  \cdots & a_{1q}\mathbf{B}\\
	\vdots&  \ddots &\vdots\\
	a_{p1}\mathbf{B}&  \cdots & a_{pq}\mathbf{B}
\end{bmatrix}
$$
and by $\oplus$ the Kronecker sum defined by $\mathbf{A}\oplus\mathbf{B}=A\otimes\mathbf{I}+\mathbf{I}\otimes B$. We define the vec-operator applied on a matrix $\mathbf{A}$ denoted by $\text{vec}(\mathbf{A})$, which stacks its columns into a column vector. Some of associated matrix properties used in this paper are given, for all matrices with suitable dimensions $\mathbf{A},\mathbf{B}, \mathbf{C}$ and $\mathbf{D}$, by
\begin{equation}
(\mathbf{A}\otimes\mathbf{C})(\mathbf{B}\otimes\mathbf{D})=\mathbf{A}\mathbf{B}\otimes\mathbf{C}\mathbf{D},\quad 
\text{vec}(\mathbf{A}\mathbf{B}\mathbf{C})=(\mathbf{C}^{\mathsf{T}}\otimes\mathbf{A})\text{vec}(\mathbf{B}),\quad
e^{\mathbf{A}\oplus\mathbf{B}}=e^{\mathbf{A}}\otimes e^{\mathbf{B}},
\label{matrix property}
\end{equation}
(for more details, see, e.g., \cite{matrix cookbook} and \cite{Horn matrix}). We use the notations $H_{\mathbf{x}}(f)$ and $\nabla_{\mathbf{x}}f$ for the Hessian matrix and the gradient column vector with respect to the parameter vector ${\mathbf{x}}$ of a function $f\in \mathcal{C}^2(\mathcal{D},\R)$ which denotes the set of twice continuously differentiable real-valued functions on $\mathcal{D}$. By $\mathcal{C}_c^2(\mathcal{D},\R)$, we denote the set of functions in $\mathcal{C}^2(\mathcal{D},\R)$ with compact support. Eventually, we will denote the convergence in probability, in distribution and almost surely  by $\stackrel{\mathbb{P}}\longrightarrow$, $\stackrel{\mathcal{D}}{\longrightarrow}$, $\stackrel{a.s.}{\longrightarrow}$ respectively.
Let $(\Omega,\mathcal{F}, \mathbb{P})$ be a probability space endowed with the augmented filtration $(\mathcal{F}_t)_{t \in \R_{+}}$ corresponding to $(B_t)_{t\in\R_+}$ and a given initial value $Z_0=(Y_0,X_0)^{\mathsf{T}}$ being independent of $(B_t)_{t\in\R_+}$ such that $\mathbb{P}(Y_0\in\R_{++})=1$. Note that $(\mathcal{F}_t)_{t\in\R_+}$ satisfies the usual conditions, i.e., the filtration 
$(\mathcal{F}_{t})_{t \in \R_{+}}$ is right-continuous and $\mathcal{F}_{0}$ contains all the $\mathbb{P}$-null sets in $\mathcal{F}$. For all $t\in\R_+$, we use the notation  $\mathcal{F}_t^Y:=\sigma(Y_s;\,0\leq s\leq t)$ for the $\sigma$-algebra generated by $(Y_s)_{s\in[0,t]}$.\\

Secondly, in the statistical part, we write the model $\eqref{model01}$ using a matrix representation as follows
\begin{equation}
	\mathrm{d}Z_t=\Lambda(Z_t)\tau \mathrm{d}t+\sqrt{Z_t^1}\rho\,\mathrm{d}B_t,\quad  t \in \mathbb{R}_{+},
	\label{model0 Z}
\end{equation} 
	where $Z_t=(Z_t^1,\ldots,Z_t^d)^{\mathsf{T}}=(Y_t,X_t^{1},\ldots,X_t^{n})^{\mathsf{T}}$, for all $t\in\R_{+}$, $
\Lambda(Z_t)=\begin{bmatrix}
	\Lambda_1(Z_t)&\mathbf{0}_{1,n(d+1)}\\
	\mathbf{0}_{n,2}&\mathbf{I}_n \otimes K(Z_t)
\end{bmatrix},$
with $\Lambda_1(Z_t)=(1,-Z_t^1)$ and $K(Z_t)=(1,-Z_t^1,\ldots,-Z_t^d)$ and the $d^2+1$-dimensional vector $\tau$ is stacking in turn the unknown drift parameters of $Y,X^{1},\ldots,X^{n}$ into a column vector, namely, $$\tau=\left(a,b,m_1,\kappa_1,\theta_{11},\ldots,\theta_{1n},\ldots,m_n,\kappa_n,\theta_{n1},\ldots,\theta_{nn}\right)^{\mathsf{T}}.$$

For the third part of this section, we come back to the first representation of the model given by relation \eqref{model01} and we present a classification result of the process $Z=(Y,X)^{\mathsf{T}}$ related to its first moment.

\begin{proposition}
	Let $a\in\R_{++}$, $b\in\R$, $m,\kappa\in\R^n$ and let $\theta$ be a real diagonalizable matrix in $\mathcal{M}_n$. Suppose that $\theta$ is either a positive definite, a negative definite or a zero matrix with eigenvalues different to $b$ or all equal to $b$. Let $Z=(Y,X)^\mathsf{T}$ be the unique strong solution of the SDE $\eqref{model01}$ and suppose that $Z_0$ is integrable. Then, for all $t\in\R_{+},$ we have\\
\begin{enumerate}
	\item \label{b_positif}	for $b\in\R_{++}$,
		$
		\mathbb{E}\left(Y_{t}\right)=
			\frac{a}{b}+\mathbf{O}(e^{-b t})
	$
and 
	\begin{equation*}
	\mathbb{E}(X_t)=
	\begin{cases}
		\theta^{-1}m-\dfrac{a}{b}\theta^{-1}\kappa+\mathbf{O}(e^{-(\lambda_{\min}(\theta)\wedge b)t})\mathbf{1}_n,&\lambda_{\min}(\theta)\in\R_{++},\\
		t\left( m-\dfrac{a}{b}\kappa\right) +\mathbf{O}(1)\mathbf{1}_n,&\lambda_{\min}(\theta)=\lambda_{\max}(\theta)=0,\\
		e^{-t \theta}\left(\mathbb{E}(X_0)(\theta-b \mathbf{I}_n)^{-1}\kappa+\Xi_0\right)+\mathbf{O}(1)\mathbf{1}_n,&\lambda_{\max}(\theta)\in\R_{--},		
	\end{cases}	
	\label{expect b>0}~~~~~~~~~~~~~~~~~~~~~~~~~~~~~~~~~~~~~~~~~~~~~~~
\end{equation*}
\item \label{b_nul} for $b=0$, 
	$
	\mathbb{E}\left(Y_{t}\right)=
		at+\mathbf{O}(1)
$
and 
\begin{equation*}
	\mathbb{E}(X_t)=
	\begin{cases}
		-ta\theta^{-1}\kappa +\mathbf{O}(1)\mathbf{1}_n,&\lambda_{\min}(\theta)\in\R_{++},\\
		-\dfrac{t^2}{2}a\kappa+\mathbf{O}(t)\mathbf{1}_n,&\lambda_{\min}(\theta)=\lambda_{\max}(\theta)=0,\\
		e^{-t \theta}\left(\mathbb{E}(Y_0)\theta^{-1}\kappa+\mathbb{E}(X_0)-\theta^{-1}m-a\theta^{-2}\kappa\right)+\mathbf{O}(t)\mathbf{1}_n,&\lambda_{\max}(\theta)\in\R_{--},		
	\end{cases}	
	\label{expec b=0}~~~~~~~~~~~~~~~~~~~~~~~~~~~~~~~~
\end{equation*}
\item \label{b_negatif} for $b\in\R_{--}$,
	$
	\mathbb{E}\left(Y_{t}\right)=
		\left(\mathbb{E}(Y_0)-\dfrac{a}{b}\right)e^{-bt}+\mathbf{O}(1)
$
and
\begin{equation*}
	\mathbb{E}(X_t)=
	\begin{cases}
		e^{-bt}\left(\left(\frac{a}{b}-\mathbb{E}\left(Y_{0}\right)\right)\left(\theta-b \mathbf{I}_{n}\right)^{-1} \kappa\right)+\mathbf{O}(1)\mathbf{1}_n,&\lambda_{\min}(\theta)\in\R_{++},\\
		e^{-bt}\left( \dfrac{1}{b}\mathbb{E}(Y_0)\kappa+\mathbb{E}(X_0)-\dfrac{a}{b^2}\kappa\right) +\mathbf{O}(t)\mathbf{1}_n,&\lambda_{\min}(\theta)=\lambda_{\max}(\theta)=0,\\
		e^{-bt}\left(-\mathbb{E}(Y_0)(\theta-b\mathbf{I}_n)^{-1}\kappa+\dfrac{a}{b}(\theta-b\mathbf{I}_n)^{-1}\kappa \right)+\mathbf{O}(e^{-\lambda_{\min}(\theta) t})\mathbf{1}_n ,&b<\lambda_{\min}(\theta)\leq\lambda_{\max}<0,\\
		te^{-bt}\left( -\mathbb{E}(Y_0)\kappa+\dfrac{a}{b}\kappa\right) +\mathbf{O}(e^{-bt})\mathbf{1}_n ,&\lambda_{\min}(\theta)=\lambda_{\max}(\theta)=b,\\
		e^{-t \theta}\left( \mathbb{E}(Y_0)(\theta-b\mathbf{I}_n)^{-1}\kappa+\Xi_0\right)+\mathbf{O}(e^{-bt})\mathbf{1}_n ,&\lambda_{\max}(\theta)<b,
	\end{cases}	
\end{equation*}
where $\Xi_0=\mathbb{E}(X_0) -\theta^{-1}m+\dfrac{a}{b}\theta^{-1}\kappa-\dfrac{a}{b}(\theta-b\mathbf{I}_n)^{-1}\kappa.$

\end{enumerate}
	\label{classification}
\end{proposition}
\begin{proof}
First  by applying Itô's formula on the processes $(e^{bt}Y_t)_{t\in\R_+}$ and $(e^{t \theta}X_t)_{t\in\R_+}$, we obtain, for $t\in\R_+$,
\begin{equation}
	\begin{aligned}
		Y_t &= e^{-bt}Y_0+a\displaystyle\int_0^t e^{-b(t-s)}\mathrm{~d}s+\rho_{11} \displaystyle\int_0^t e^{-b(t-s)}\sqrt{Y_s}\mathrm{~d}B^1_s.
	\end{aligned}
\label{Y expression}
\end{equation}
and
\begin{equation}
	\begin{aligned}
		X_t= e^{-t \theta}X_0+\displaystyle\int_0^t e^{-(t-s) \theta}m\mathrm{~d}s-\displaystyle\int_0^t e^{-(t-s) \theta} Y_s\kappa\mathrm{~d}s +\displaystyle\int_0^t \sqrt{Y_s} e^{-(t-s) \theta} \tilde{\rho}\mathrm{~d}B_s.
	\end{aligned}
\label{X expression}
\end{equation}
By taking the expectation in the two sides of equtions \eqref{Y expression} and \eqref{X expression} and thanks to Fubini-Tonelli it is easy to check that
\begin{equation}
\mathbb{E}\left(Y_{t}\right)=e^{-b t} \mathbb{E}\left(Y_{0}\right)+a \int_{0}^{t} e^{-b u} \mathrm{~d} u= \begin{cases}\frac{a}{b}+\left(\mathbb{E}\left(Y_{0}\right)-\frac{a}{b}\right) e^{-b t}, & b \neq 0,\\
\mathbb{E}\left(Y_{0}\right)+a t, & b=0,\end{cases}
\label{E(Y_t)}
\end{equation}
and
\begin{equation}
	\mathbb{E}\left(X_{t}\right) 
	=e^{-t \theta} \mathbb{E}\left(X_{0}\right)+\int_{0}^{t} e^{-s \theta} m \mathrm{~d} s-\int_{0}^{t} e^{-(t-s) \theta}\left(e^{-b s} \mathbb{E}\left(Y_{0}\right)+a \int_{0}^{s} e^{-b u} \mathrm{~d} u\right) \kappa \mathrm{~d} s.
	\label{X expec}
\end{equation}
Thanks to $\eqref{E(Y_t)}$, it is easy to get the behavior of $\mathbb{E}(Y_t)$ for the three cases ($b\in\R_{++},\ b=0$ and $b\in\R_{--}$). Concerning $\mathbb{E}(X_t)$ we distinguish two cases. For $b=0$, if $\lambda_{\min}(\theta)=\lambda_{\max}(\theta)=0$, the relation \eqref{X expec} becomes  
$\mathbb{E}(X_t)=		\mathbb{E}\left(X_{0}\right)+t\left(m-\mathbb{E}(Y_0)\kappa\right)-\dfrac{1}{2}t^2a\kappa$ and we obtain the second assertion of item \ref{b_nul}. Otherwise, we have $\mathbb{E}(X_t)=e^{-t \theta}\mathbb{E}\left(X_{0}\right)+\theta^{-1}\left(I_{n}-e^{-t \theta}\right) m-\left(\mathbb{E}\left(Y_{0}\right)\left(1-e^{-t\theta}\right)+at\right) \theta^{-1} \kappa+a \theta^{-2}\left(I_{n}-e^{-t \theta}\right) \kappa$. Its behavior depends on the sign of the eigenvalues of $\theta$ and it is easy to obtain the first and the third assertions of item \ref{b_nul}.
Now for $b\neq0$, if $\lambda_{\min}=\lambda_{\max}=0$, the relation \eqref{X expec} becomes $\mathbb{E}(X_t)=\mathbb{E}(X_0)+t\left(m-\dfrac{a}{b}\kappa\right)+\left(\dfrac{1}{b}\mathbb{E}(Y_0)-\dfrac{a}{b^2}\right)(e^{-bt}-1)\kappa$ and the behavior depends on the sign of $b$: for $b\in\R_{++}$, we obtain the second assertion of item \ref{b_negatif} and for $b\in\R_{--}$, we obtain the second assertion of item \ref{b_negatif}, else if $\lambda_{\min}=\lambda_{\max}=b$, the relation \eqref{X expec} becomes $\mathbb{E}(X_t)=e^{-bt}\mathbb{E}\left(X_{0}\right)+\frac{1}{b}\left(1-e^{-bt}\right)m-\frac{a}{b^2}\left(1-e^{-bt}\right)\kappa+\left(\frac{a}{b}-\mathbb{E}\left(Y_{0}\right)\right) t e^{-b t}\kappa$ and we obtain a part of the first assertion of item \ref{b_positif} and the fourth assertion of item \ref{b_negatif}, otherwise, $\mathbb{E}(X_t)=e^{-t \theta} \mathbb{E}\left(X_{0}\right)+\left(\frac{a}{b}-\mathbb{E}\left(Y_{0}\right)\right)\left(\theta-b \mathbf{I}_{n}\right)^{-1}\left(e^{-b t \mathbf{I}_{n}}-e^{-t \theta}\right) \kappa+\theta^{-1}\left(\mathbf{I}_{n}-e^{-t \theta}\right) \left(m-\frac{a}{b}\kappa\right)$ which covers the rest of cases in item \ref{b_positif} and item \ref{b_negatif}. This completes the proof of Proposition \ref{classification}.
\end{proof}
\begin{remarque}
		If $\theta$ is an arbitrary matrix in $\mathcal{M}_n$ admitting joint positive, negative or zero eigenvalues, then
	\begin{equation}
		\mathbb{E}(X_t)=(QP)^{-1}\mathbb{E}(\Upsilon_t),
		\label{E(X_t) mixte}
	\end{equation}
	for all $t\in\R_+$, where $\Upsilon_t=QPX_t$ with $P\in\mathcal{M}_n$ is the change-of-basis matrix transforming $\theta$ into a diagonal matrix $D$, $Q=(Q_1,Q_2,Q_3,Q_4,Q_5)^{\mathsf{T}}\in \mathcal{M}_n$ is a permutation matrix (a non-singular square binary matrix, called also a doubly stochastic matrix, that allows us to rearrange the eigenvalues of $\theta$ on the diagonal of the matrix $D$) such that, for all $i,j\in\lbrace1,\ldots,5\rbrace$, the rows of $Q_i$ and the rows of $Q_j$ are orthogonal when $i\neq j$. Note that for all $k\in\lbrace1,\ldots,5\rbrace$, $Q_k\in\mathcal{M}_{n_k,n}$ where $n_k\in\N$ satisfying $\sum_{k=1}^5 n_k=n$, and
	$\Upsilon_t=(\Upsilon^1_t,\Upsilon^2_t,\Upsilon^3_t,\Upsilon^4_t,\Upsilon^5_t)^{\mathsf{T}}$ with $\Upsilon^k_t\in\R^{n_i}$ is the solution of the following SDE
	\begin{equation*}
		\mathrm{d}\Upsilon^k_{t}=\left(Q_k P m-Q_k P\kappa Y_{t}-\tilde{D}_k \Upsilon_{t}^k\right) \mathrm{~d} t+\sqrt{Y_{t}} Q_k P \tilde{\rho} \mathrm{~d} B_{t}, \quad t\in\R_+,
	\end{equation*}
	where $\tilde{D}_k \in\mathcal{M}_{n_k}$ satisfies $\lambda_{min}(\tilde{D}_1)\in\R_{++}$, $\lambda_{min}(\tilde{D}_2)=\lambda_{max}(\tilde{D}_2)=0$, $\lambda_{max}(\tilde{D}_3)\in \R_{--}$ with $\lambda_{\min}(\tilde{D}_3)>b$ when $b\in\R_{--}$, $\lambda_{max}(\tilde{D}_4)\in\R_{--}$ with $\lambda_{min}(\tilde{D}_4)=\lambda_{max}(\tilde{D}_4)=b$ when $b\in\R_{--}$ and $\lambda_{max}(\tilde{D}_5)\in\R_{--}$ with $\lambda_{\max}(\tilde{D}_5)<b$ when $b\in\R_{--}$.
	Consequently, $\mathbb{E}(\Upsilon_t)=(\mathbb{E}(\Upsilon_t^1),\ldots,\mathbb{E}(\Upsilon_t^5))^{\mathsf{T}}$ and for each $k\in\lbrace1,\ldots,5\rbrace$, $\mathbb{E}(\Upsilon_t^k)$ satisfies one of the items \ref{b_positif}, \ref{b_nul} or \ref{b_negatif} by replacing $m,\kappa$ and $\theta$ by $Q_k Pm,Q_k P\kappa$ and $\tilde{D}_k$, respectively.
In fact, in this case, we can choose a permutation matrix $Q\in\mathcal{M}_n$ and a change-of-basis matrix $P$ to diagonalize $\theta$ as follows 
\begin{equation*}
	\theta=(QP)^{-1}\tilde{D}(Q P)=P^{-1}\begin{bmatrix}
		Q_1\\
		Q_2\\
		Q_3\\
		Q_4\\
		Q_5
	\end{bmatrix}^{-1}\begin{bmatrix}
		\tilde{D}_1&\mathbf{0}&\mathbf{0}&\mathbf{0}&\mathbf{0}\\
		\mathbf{0}&\tilde{D}_2&\mathbf{0}&\mathbf{0}&\mathbf{0}\\
		\mathbf{0}&\mathbf{0}&\tilde{D}_3&\mathbf{0}&\mathbf{0}\\
		\mathbf{0}&\mathbf{0}&\mathbf{0}&\tilde{D}_4&\mathbf{0}\\
		\mathbf{0}&\mathbf{0}&\mathbf{0}&\mathbf{0}&\tilde{D}_5
	\end{bmatrix} \begin{bmatrix}
		Q_1\\
		Q_2\\
		Q_3\\
		Q_4\\
		Q_5
	\end{bmatrix}P.
\end{equation*}
\end{remarque}
Based on the asymptotic behavior of the expectation $\mathbb{E}(Z_t)$ as $t$ tends to $\infty$ given in Proposition \ref{classification}, we introduce a classification of $\mathit{AD}(1,n)$ processes defined by the SDE $\eqref{model0 Z}$.
\begin{definition}
Let $(Z_t)_{t\in\R_{+}}$ be the unique strong solution of $\eqref{model0 Z}$ satisfying $\mathbb{P}(Z_0^1\in\R_{++})=1$. Suppose that $\theta$ is either a positive definite, a negative definite or a zero matrix with eigenvalues different to $b$ or all equal to $b$. We call $(Z_t)_{t\in\R_+}$ subcritical if $b\wedge \lambda_{\min}(\theta)\in\R_{++}$, i.e., when $\mathbb{E}(Z_t)$ converges as $t\to\infty$, critical if either $b\in\R_+$ and $\lambda_{\min}(\theta)=\lambda_{\max}(\theta)=0$ or $b=0$ and $\lambda_{\min}(\theta)\in\R_{++}$, i.e., when $\mathbb{E}(Z_t)$ has a polynomial expansion and supercritical if $b\wedge\lambda_{\max}(\theta)\in\R_{--}$, i.e., when $\mathbb{E}(Z_t)$ has an exponential expansion.
\end{definition}
In the following section we state the stationarity theorem of $(Z_t)_{t\in\R_+}$ defined by the last $\mathit{AD}(1,n)$ model $\eqref{model0 Z}$ and its proof.
\section{Stationarity}
In a special case where $n=1$, the study of existence of stationary distributions for $\mathit{AD}(1,1)$ process was treated by Barczy M. et al. in \cite{Barczy stationarity ergodicty} with $\kappa=0$ and it is called affine two-factor model. Analogously, it was proved by Bolyog B. and Pap G. in \cite{Boylog1} for affine two-factor diffusions where $\kappa\in\R$. By the following theorem, we extend these studies for the $\mathit{AD}(1,n)$ process $Z=(Y,X)^{\mathsf{T}}$ defined by the model $\eqref{model01}$. 
\begin{theoreme}
Let us consider the $\mathit{AD}(1,n)$ model $\eqref{model01}$ with $a\in\R_+,\ b\in\R_{++},\ m\in\R^n,\ \kappa\in\R^n$, $\theta\in\mathcal{M}_n$ a diagonalizable positive definite matrix and $Z_0=(Y_0,X_0)^{\mathsf{T}}$ is a random initial value independent of $(B_t)_{t\in\R_+}$ satisfying $\mathbb{P}(Y_0\in\R_+)=1$.
\begin{enumerate}[label=\arabic*)]
\item Then, $Z_t\stackrel{\mathcal{D}}{\longrightarrow}Z_{\infty}$ as $t\to\infty$, where $Z_{\infty}=(Y_\infty,X_\infty)^{\mathsf{T}}$ and the distribution of $Z_\infty$ is given by
\begin{equation}
\mathbb{E}\left(e^{\nu^{\mathsf{T}} Z_\infty}\right)=\mathbb{E}\left(e^{-\lambda Y_\infty +i \mu^{\mathsf{T}} X_\infty}\right)=\exp\left(a\displaystyle\int_0^{\infty} \mathcal{K}_s\left(-\lambda,\mu \right)\mathrm{~d}s+i\mu^{\mathsf{T}}\theta^{-1} m\right),
\label{loi limite Stationarity theorem}
\end{equation}
for $\nu=(-\lambda,i\mu)\in\mathcal{U}_1\times \mathcal{U}_2$, where $\mathcal{U}_1:=\lbrace u\in \mathbb{C} : \texttt{Re}(u)\in\R_{-}\rbrace$ and $\mathcal{U}_2:=\lbrace u\in \mathbb{C}^n : \texttt{Re}(u)=0\rbrace$ and for all $(u_1,u_2)\in \mathcal{U}_1\times\R^n$, the function $t\mapsto \mathcal{K}_t(u_1,u_2)$ is the unique solution of the following (deterministic) non-linear Riccati partial differential equation (PDE)
\begin{equation}
\begin{cases}
\dfrac{\partial \mathcal{K}_t}{\partial t}(u_1,u_2)&=
	\dfrac{\rho_{11}^2}{2}\mathcal{K}_t^2(u_1,u_2)-\left(b-i\rho_{11}\rho_{J1}^{\mathsf{T}}e^{-t\theta^{\mathsf{T}}}u_2\right) \mathcal{K}_t(u_1,u_2) -i\kappa^{\mathsf{T}}\,e^{-t\theta^{\mathsf{T}}}u_2 \\&\quad-\frac{1}{2}\left(\text{vec}(\rho_{JJ}\rho_{JJ}^{\mathsf{T}})\right)^\mathsf{T}\,e^{-t (\theta^\mathsf{T}\oplus\theta^\mathsf{T})}(u_2\otimes u_2)-\dfrac{1}{2}\left(\rho_{J1}^{\mathsf{T}}e^{-t\theta^{\mathsf{T}}}u_2\right)^2\\ 
	\mathcal{K}_0(u_1,u_2)&=u_1.
\end{cases}
\label{EDP_loi limite Stationarity theorem}
\end{equation} 
\label{i Stationarity}
\item In addition, if $Z_0$ has the same distribution as $Z_{\infty}$ given by $\eqref{loi limite Stationarity theorem}$, then $(Z_t)_{t\in\R_+}$ is strictly stationary.
\end{enumerate}
\label{Stationarity theorem}
\end{theoreme}
In order to prove this theorem we consider the following technical lemma: 
\begin{lemme}
 We can assume, without loss of generality, that $\rho_{J1}=\mathbf{0}_n$ in Theorem \ref{Stationarity theorem}.
 \label{Stationarity_Lemma}
\end{lemme}
The proof of this Lemma will be postponed after the theorem proof.
\begin{proof}(Proof of  $(Z_t)_{t\in[0,T]}$, for some $T\in\R_{++}$ $(Z_t)_{t\in[0,T]}$, for some $T\in\R_{++}$Theorem \ref{Stationarity theorem})
 According to Lemma \ref{Stationarity_Lemma}, we can assume that $\rho_{J1}=\mathbf{0}_{n}$ and hence consider the model
	\begin{equation}
		\begin{cases}
			\begin{aligned}
				\mathrm{d}Y_t &=(a-b Y_t) \mathrm{~d}t+ \rho_{11}\sqrt{Y_t} \mathrm{~d}B^1_t,\\
				\mathrm{d}X_t &=(m-\kappa Y_t-\theta X_t) \mathrm{~d}t + \sqrt{Y_t}\, \rho_{JJ} \mathrm{~d}W_t,\quad \text{where } W_t=(B_t^2,\ldots,B_t^{d}).
			\end{aligned}
		\end{cases}
		\label{EDSSSS}
	\end{equation}.
\begin{enumerate}[label=\arabic*)]
\item  The proof of the first assertion follows four main steps:\\ \textbf{Step 1}:
the first step is devoted to convert the asymptotic study for a triplet process $\tilde{Z}=(Y,U,V)$ instead of $Z=(Y,X)$, where, for all $t\in\R_+$,\\
$$U_t:=\int_0^t Y_s e^{-(t-s) \theta} \kappa\mathrm{~d}s\quad \text{  and  }\quad V_t:= \text{vec}\left(\int_0^t Y_s e^{-(t-s) \theta} \rho_{JJ} \rho_{JJ}^{\mathsf{T}} e^{-(t-s)\theta^{\mathsf{T}}}\mathrm{~d}s\right),$$
and to write its Fourier-Laplace transform. For $(\lambda,\mu),(y_0,x_0)\in\mathcal{D}$ and $t\in\R_+$, by the explicit expression of $X_t$ given by $\eqref{X expression}$ and the tower property, the conditional Fourier-Laplace transform of $(Y_t,X_t)$, namely $\mathbb{E}\left(e^{-\lambda Y_t +i \mu^{\mathsf{T}} X_t}\displaystyle\vert(Y_0,X_0)=(y_0,x_0)\right)$ is equal to
{\small
\begin{equation}
\mathbb{A}_t(\mu,\theta,m,x_0)\mathbb{E}\left(\exp\left({-\lambda Y_t-i\mu^{\mathsf{T}} U_t}\right)\,\mathbb{E}\left(\exp\left({i\mu^{\mathsf{T}} \displaystyle\int_0^t \sqrt{Y_s} e^{-(t-s) \theta} \rho_{JJ}\mathrm{~d}W_s}\right)\mid \mathcal{F}_t^Y\right)\displaystyle\vert Y_0=y_0\right),
\label{d2}
\end{equation}
where $\mathbb{A}_t(\mu,\theta,m,x_0)=\exp\left({i\mu^{\mathsf{T}} e^{-t \theta}x_0+i \mu^{\mathsf{T}}\theta^{-1} m-i \mu^{\mathsf{T}}\theta^{-1}e^{-t \theta}  m }\right)$.
Further, by the independence of $Y$ and $W$, the second property on the Kronecker operators for matrices in \eqref{matrix property} and the definition of $V_t$, we deduce that the conditional expectation $\mathbb{E}\left(\exp\left({i\mu^{\mathsf{T}} \displaystyle\int_0^t \sqrt{Y_s} e^{-(t-s) \theta} \rho_{JJ}\mathrm{~d}W_s}\right)\mid \mathcal{F}_t^Y\right)$ is equal to
{\small
\begin{equation}
\begin{aligned}
\mathbb{E}\left(\exp\left(-\frac{1}{2}\mu^{\mathsf{T}} \displaystyle\int_0^t Y_s e^{-(t-s) \theta} \rho_{JJ} \rho_{JJ}^{\mathsf{T}} e^{-(t-s)\theta^{\mathsf{T}}}\mathrm{~d}s\,\mu\right)\right)
=\mathbb{E}\left(\exp\left(-\frac{1}{2}(\mu^{\mathsf{T}}\otimes\mu^{\mathsf{T}})V_t\right)\right).
\end{aligned}
\label{d2prime}
\end{equation}}
\normalsize
By combining $\eqref{d2}$ and $\eqref{d2prime}$ together we obtain
\begin{equation}
	\mathbb{E}\left(e^{-\lambda Y_t +i \mu^{\mathsf{T}} X_t}\displaystyle\vert(Y_0,X_0)=(y_0,x_0)\right)=\mathbb{A}_t(\mu,\theta,m,x_0)\; \mathbb{E}\left(e^{-\lambda Y_t-i\mu^{\mathsf{T}} U_t-\frac{1}{2}(\mu^{\mathsf{T}}\otimes\mu^{\mathsf{T}})V_t}\displaystyle\vert Y_0=y_0\right)
	\label{(Y X) to (Y U V)}.
\end{equation}
Note that the process $\tilde{Z}=(Y,U,V)$ introduced above is solution to
\begin{equation*}
\begin{aligned}
\begin{cases}
\mathrm{d}Y_t &=(a-b Y_t) \mathrm{~d}t+\rho_{11} \sqrt{Y_t}\mathrm{~d}B^1_t,\\
\mathrm{d}U_t &=(Y_t\kappa-\theta U_t) \mathrm{~d}t,\\
\mathrm{d}V_t &=\left(Y_t \text{vec}(\rho_{JJ}\rho_{JJ}^{\mathsf{T}})- (\theta\oplus\theta) V_t \right) \mathrm{~d}t,
\end{cases} t\in\R_{+},
\end{aligned}
\end{equation*}
with initial values $Y_0$, $\mathbf{0}_{n}$ and $\mathbf{0}_{n^2}$, which is also an affine process, see e.g. Filipovic in \cite[Theorem 3.2]{Filipovic}. Hence, according to Definition 2.1 in \cite{Filipovic},  we have for all $T\in\R_+$, $ 0\leq t\leq T$ and $u=(u_1,u_2,u_3)\in \mathcal{U}:=\mathcal{U}_1\times \mathcal{U}_2 \times \mathcal{U}_3$, where $\mathcal{U}_3:=\lbrace u\in\R^{n^2}\,\vert\,\exists\tilde{u}\in\mathcal{U}_2;\,u=\tilde{u}\otimes\tilde{u}\rbrace$, the following affine transform formula
\begin{equation}
\begin{aligned}
\mathbb{E}\left(e^{u^{\mathsf{T}}\tilde{Z}_T}\mid \mathcal{F}_t\right)=\mathbb{E}\left(e^{u_1 Y_T +u_2^\mathsf{T} U_T +u_3^\mathsf{T} V_T}\mid \mathcal{F}_t\right)=\exp\left(\phi(T-t,u)+\psi(T-t,u)^{\mathsf{T}}\tilde{Z}_t\right),
\end{aligned}
\label{Laplace Y,U,V}
\end{equation}
where  $\psi=(\psi_1,\psi_2,\psi_3)^{\mathsf{T}}$ is the solution of the following Riccati system
\begin{equation}
	\begin{cases}
		\dfrac{\partial \psi_1}{\partial t}(t,u)&=\frac{\rho_{11}^2}{2}\psi_1^2(t,u)-b\psi_1(t,u)+\kappa^\mathsf{T} \psi_2(t,u)+(\text{vec}(\rho_{JJ}\rho_{JJ}^{\mathsf{T}}))^\mathsf{T}\psi_3(t,u),\\[5pt]
		\dfrac{\partial \psi_2}{\partial t}(t,u)&=-\theta^{\mathsf{T}} \psi_2(t,u),\\[5pt]
\dfrac{\partial \psi_3}{\partial t}(t,u)&=-(\theta^{\mathsf{T}}\oplus \theta^{\mathsf{T}}) \psi_3(t,u),\\[5pt]
		\psi_1(0,u)&=u_1,\ \psi_2(0,u)=u_2,\ \psi_3(0,u)=u_3,
	\end{cases}
	\label{psi10}
\end{equation}
and $\phi$ is defined, for all $u\in \mathcal{U}$ and $t\in\R_{+}$, by
$		\phi(t,u)=a\displaystyle\int_0^t \psi_1(s,u)\mathrm{~d}s.$
Using the explicit solutions $\psi_2(t,u)=e^{-t\theta^{\mathsf{T}}}u_2$ and
$\psi_3(t,u)=e^{-t (\theta^{\mathsf{T}}\oplus\theta^{\mathsf{T}})\, }u_3$, we deduce that $\psi_1$ solves the followig PDE
\begin{equation}
\begin{cases}
\dfrac{\partial \psi_1}{\partial t}(t,u)&=\frac{\rho_{11}^2}{2}\psi_1^2(t,u)-b\psi_1(t,u)+\kappa^\mathsf{T}\,e^{-t\theta^\mathsf{T}}u_2 +(\text{vec}(\rho_{JJ}\rho_{JJ}^{\mathsf{T}}))^\mathsf{T}\,e^{-t (\theta^{\mathsf{T}}\oplus\theta^{\mathsf{T}})}u_3,\\
\psi_1(0,u)&=u_1.
\end{cases}
\label{psi1}
\end{equation}
for $u\in \mathcal{U}$. Consequently, by taking $t=0$ in $\eqref{Laplace Y,U,V}$, we obtain
\begin{equation}
\begin{aligned}
\mathbb{E}\left(e^{u_1 Y_t + u_2^{\mathsf{T}} U_t + u_3^{\mathsf{T}} V_t}\displaystyle\vert Y_0= y_0  \right)
&=\exp\left(a\displaystyle\int_0^t \psi_1(s,u)\mathrm{~d}s+y_0 \psi_1(t,u)\right).
\end{aligned}
\label{transfo}
\end{equation}
\textbf{Step 2}: we prove that if $u_1\in\mathcal{U}_1$, we have $\psi_1(t,u)\in\mathcal{U}_1$, for all $t\in\R_+$ and $(u_2,u_3)\in \mathcal{U}_2\times \mathcal{U}_3$. Indeed, using the fact that  $u_3^{\mathsf{T}} V_t\in\R_{-}$ since
$$u_3^{\mathsf{T}} V_t= -\dfrac{1}{2} (\mu^{\mathsf{T}}\otimes\mu^{\mathsf{T}}) V_t = -\dfrac{1}{2}\mu^{\mathsf{T}} \displaystyle\int_0^t Y_s e^{-(t-s) \theta} \rho_{JJ} \rho_{JJ}^{\mathsf{T}} e^{-(t-s)\theta^{\mathsf{T}}}\mathrm{~d}s\,\mu  $$
and $\displaystyle\int_0^t Y_s e^{-(t-s) \theta} \rho_{JJ} \rho_{JJ}^{\mathsf{T}} e^{-(t-s)\theta^{\mathsf{T}}}\mathrm{~d}s$ is a positive semi-definite matrix, we deduce that 
\begin{equation}
\begin{aligned}
\left\vert\mathbb{E}\left(e^{u_1 Y_t + u_2^{\mathsf{T}} U_t + u_3^{\mathsf{T}} V_t}\displaystyle\vert Y_0=y_0\right)\right\vert\leq \mathbb{E}\left(e^{\texttt{Re}(u_1) Y_t+u_3^{\mathsf{T}} V_t}\displaystyle\vert Y_0=y_0\right)\leq 1,
\end{aligned}
\label{E<1}
\end{equation}
for $t\in\R_+$ and $u\in \mathcal{U}$. Consequently, combining the equations $\eqref{transfo}$ and $\eqref{E<1}$, we obtain
\begin{equation*}
\begin{aligned}
\left\vert \exp\left( a\int_0^t \psi_1(s,u)\mathrm{~d}s+y_0 \psi_1(t,u)\right)\right\vert=\exp \left( a\,\texttt{Re}\left(\int_0^t \psi_1(s,u)\mathrm{~d}s\right)+y_0\,\texttt{Re}\left(\psi_1(t,u)\right)\right)\leq 1,
\end{aligned}
\end{equation*}
for all $t\in\R_+$ and $u\in \mathcal{U}$. Now, looking to the ordinary differential equation $\eqref{psi1}$, it's obvious that its solution $\psi_1$ does not depend on the parameters $a$ and $y_0$. Therefore, we are allowed to take $a=0$ and choose $y_0\in\R_{++}$ which implies that $\psi_1(t,u)\in\mathcal{U}_1$ when $u_1\in\mathcal{U}_1$, for all $t\in\R_+$ and $(u_2,u_3)\in \mathcal{U}_2\times \mathcal{U}_3$.

\textbf{Step 3}: Let $\mathcal{U}'=\mathcal{U}_1\times \R^n \times \mathcal{U}_3$. Now, we concentrate on the construction of an upper-bound for the modulus of the function $\tilde{K}$ defined through $\psi_1$ as follows  $$\R_+\times \mathcal{U}' \ni (t,u_1,u_2,u_3)\mapsto \tilde{K}_t(u_1,u_2,u_3):=\psi_1(t,(u_1,-i u_2,u_3)),$$ which satisfies, for all $u=(u_1,u_2,u_3)\in\mathcal{U}'$,
\begin{equation}
\begin{aligned}
\dfrac{\partial \tilde{K}_t}{\partial t}(u_1,u_2,u_3)&=
\frac{\rho_{11}^2}{2}\tilde{K}_t^2(u_1,u_2,u_3)-b \tilde{K}_t(u_1,u_2,u_3)-i\kappa^\mathsf{T}\,e^{-t\theta^\mathsf{T}}u_2 +(\text{vec}(\rho_{JJ}\rho_{JJ}^{\mathsf{T}}))^\mathsf{T}\,e^{-t (\theta^{\mathsf{T}}\oplus\theta^{\mathsf{T}})}u_3,\end{aligned}
\label{K tilde EDO }
\end{equation}
and
\begin{equation}
\begin{aligned}
\mathbb{E}\left(e^{u_1 Y_t - iu_2^{\mathsf{T}} U_t + u_3^{\mathsf{T}} V_t}\displaystyle\vert Y_0=y_0\right)=\exp\left(\tilde{g}_t(u)+y_0 \tilde{K}_t(u)\right),
\end{aligned}
\label{Lap en K}
\end{equation}
where $\tilde{g}_t(u):=a\displaystyle\int_0^t \tilde{K}_s(u)\mathrm{~d}s$. Let  the real functions $\R_+\times \mathcal{U}'\ni (t,u)\mapsto v_t(u):=- \texttt{Re}(\tilde{K}_t(u))$ and $\R_+\times \mathcal{U}'\ni (t,u)\mapsto w_t(u):= \texttt{Im}(\tilde{K}_t(u))$
satisfying the following Riccati system
\begin{equation}
\begin{cases}
\dfrac{\partial v_t}{\partial t}(u)&=-\frac{\rho_{11}^2}{2}v_t^2(u)-bv_t(u)+\frac{\rho_{11}^2}{2}w_t^2(u)-(\text{vec}(\rho_{JJ}\rho_{JJ}^{\mathsf{T}}))^\mathsf{T}\,e^{- t(\theta^{\mathsf{T}}\oplus\theta^{\mathsf{T}})}u_3,\\[5pt]
\dfrac{\partial w_t}{\partial t}(u)&=-(b+\rho_{11}^2 v_t(u)) w_t(u)-\kappa^{\mathsf{T}}e^{-t\theta^\mathsf{T}}u_2,\\[5pt]
v_0(u)&=-\texttt{Re}(u_1),\ w_0(u)=\texttt{Im}(u_1),

\end{cases}\quad t\in\R_{+},
\label{vw}
\end{equation}

and a general solution of $w_t$ takes the form
\begin{equation*}
w_t(u)=Ce^{-\int_0^t f_z(u)\mathrm{~d}z}-e^{-\int_0^t f_z(u)\mathrm{~d}z}\displaystyle\int_0^t e^{\int_0^s f_z(u)\mathrm{~d}z}\kappa^{\mathsf{T}}e^{-s\theta^\mathsf{T} }u_2\mathrm{~d}s,
\end{equation*}
 with $f_t(u):=b+\rho_{11}^2v_t(u)$, for $t\in\R_{+}$ and $C$ is a real constant. Taking into account the initial value $w_0(u)=\texttt{Im}(u_1)$, we obtain $C=\texttt{Im}(u_1)$. Consequently, for all $t\in\R_+$,
\begin{equation*}
\vert w_t(u)\vert \leq \norm{u_1} e^{-\int_0^t f_z(u)\mathrm{~d}z}+  \norm{\kappa}\norm{u_2}\displaystyle\int_0^t e^{-\lambda_{\min}(\theta) s -\int_s^t f_z(u)\mathrm{~d}z}\mathrm{~d}s.
\end{equation*} 
Thanks to the fact that $f_t(u)\geq b\in\R_{+}$, for all $t\in\R_{+}$ and $u\in\mathcal{U}'$,
\begin{equation*}
\begin{aligned}
\displaystyle\int_0^t e^{-\lambda_{\min}(\theta) s -\int_s^t f_z(u)\mathrm{~d}z}\mathrm{~d}s&\leq\displaystyle\int_0^t e^{-\lambda_{\min}(\theta) s -(t-s)b}\mathrm{~d}s\\&=
\begin{cases}
\begin{aligned}
&\dfrac{e^{-\lambda_{\min}(\theta) t}-e^{-bt}}{b-\lambda_{\min}(\theta)}\leq \dfrac{e^{-(\lambda_{\min}(\theta)\wedge b)t}}{\vert b-\lambda_{\min}(\theta)\vert},\quad &b\neq\lambda_{\min}(\theta),\\
&te^{-bt}\leq e^{-\frac{b}{2}t}\underset{t\in\R_{+}}{\sup}\,te^{-\frac{b}{2}t}=\frac{2}{e\,b}e^{-\frac{b}{2}t},\quad &b=\lambda_{\min}(\theta).
\end{aligned}
\end{cases}
\end{aligned}
\end{equation*}
Thus, for all $t\in\R_+$ and $u\in \mathcal{U}'$, we obtain
\begin{equation}
\vert w_t(u)\vert \leq C_3(u)e^{-C_2 t},
\label{w ineq}
\end{equation}
with $C_3(u):=\norm{u_1}+\norm{\kappa}\norm{u_2}\left(\dfrac{1}{\vert b-\lambda_{\min}(\theta)\vert}\mathds{1}_{\lbrace b\neq\lambda_{\min}(\theta)\rbrace}+\dfrac{2}{eb}\mathds{1}_{\lbrace b=\lambda_{\min}(\theta)\rbrace}\right)$ and $C_2=\lambda_{\min}(\theta) \wedge\frac{b}{2}$.\\
Combining the first PDE of $\eqref{vw}$ with the relation $\eqref{w ineq}$, we get
$
\dfrac{\partial v_t}{\partial t}(u)\leq -b v_t(u)+C_4(u) e^{-C_2 t},$ 
with  $C_4(u)=\frac{\rho_{11}^2}{2}C_3(u)^2+\norm{\text{vec}(\rho_{JJ} \rho_{JJ}^{\mathsf{T}})}\norm{u_3}$,
and
$v_0(u)=-\texttt{Re}(u_1)$. By  the comparison theorem, we derive the inequality
$v_t(u)\leq \tilde{v}_t(u)$, for all $t\in\R_+$ and $u\in \mathcal{U}'$, where  $\tilde{v}_t$ is solution of the inhomogeneous linear differential equation
$
\dfrac{\partial \tilde{v}_t}{\partial t}(u)= -b \tilde{v}_t(u)+C_4(u) e^{-C_2 t},$
and $\tilde{v}_0(u)=-\texttt{Re}(u_1)$.
 In a similar way as in the previous calculation, we obtain
$
\tilde{v}_t(u)=-\texttt{Re}(u_1)e^{-bt}+C_4(u)e^{-bt}\displaystyle\int_0^t e^{-(b-C_2)s}\mathrm{~d}s.
$
Using the fact
$b-C_2\geq \dfrac{b}{2}\in\R_{++}$, we get
$$
0\leq v_t(u) \leq \tilde{v}_t(u)=-\textbf{Re}(u_1)e^{-bt}+C_4(u)\dfrac{e^{-C_2t}-e^{-bt}}{b-C_2}
\leq C_5(u)e^{-C_2 t},$$
with $C_5(u)=-\texttt{Re}(u_1)+\dfrac{2}{b}C_4(u)$.
Finally, we obtain
\begin{equation}
\vert \tilde{K}_t(u) \vert =\sqrt{v_t(u)^2+w_t(u)^2}\leq C_1(u)e^{-C_2t},
\label{K vers 0}
\end{equation}
for all $t\in\R_+$ and $u\in \mathcal{U}'$, with $C_1(u):=\sqrt{C_5^2(u)+C_3^2(u)}$.

\vspace*{0.009cm}

\textbf{Step 4}: In order to complete the proof of our first resut given by relations \eqref{loi limite Stationarity theorem} and \eqref{EDP_loi limite Stationarity theorem}, we prove the following assertions:
\begin{enumerate}[label=\roman*)]
\item for all $u\in\mathcal{U}'$ and $y_0\in\R_{+}$,
$\underset{t\to\infty}{\lim} \left[y_0 \tilde{K}_t(u)+\tilde{g}_t(u)\right]=a\displaystyle\int_0^\infty \tilde{K}_s(u)\mathrm{~d}s=:\tilde{g}_\infty(u).$
\label{first assertion}
\item the function $\mathcal{U}'\ni u\mapsto \tilde{g}_\infty(u)$ is continuous.
\label{second assertion}
\end{enumerate}
At fist, using the equation $\eqref{K vers 0}$, we deduce that $\underset{t\to\infty}{\lim} y_0 \tilde{K}_t(u)=0$, then thanks to Lebesgue's dominated convergence theorem, we get 
$\underset{t\to\infty}{\lim}\displaystyle\int_0^t \tilde{K}_s(u)\mathrm{~d}s =\displaystyle\int_0^\infty \tilde{K}_s(u)\mathrm{~d}s.$ Secondly, in order to prove \ref{second assertion}, let $(u^{(n)})_{n\in\N}$ a sequence in $\mathcal{U}'$ such that $\underset{n\to\infty}{\lim}u^{(n)}=u\in \mathcal{U}'$. Similarly, thanks to equation $\eqref{K vers 0}$, the continuity of $\mathcal{U}'\ni u\mapsto \tilde{K}_t(u)$ with $t\in \R_{+}$ and  the Lebesgue's dominated convergence theorem,  we deduce that 
$\underset{n\to\infty}{\lim}\displaystyle\int_0^\infty \tilde{K}_s(u^{(n)})\mathrm{~d}s=\displaystyle\int_0^\infty \tilde{K}_s(u)\mathrm{~d}s,$
which shows the continuity of $\mathcal{U}'\ni u\mapsto\tilde{g}_\infty(u)$. It is worth to note that since the function $\tilde{K}_t$ does not depend on the parameters $a$ and $y_0$ as unique solution of the differential equation  $\eqref{K tilde EDO }$, its continuity can be proved by taking $a=0$ and $y_0\in\R_{++}$ in relation \eqref{Lap en K}.

 Finally, Let us consider $(u_1,u_2,u_3)=(-\lambda,\mu,-\frac{1}{2}\mu \otimes \mu)$. Hence, thanks to the independence of $(Y_0,X_0)$ and $(B_t)_{t\in\R_{+}}$, we use first the law of total expectation, then the dominated convergence theorem and we finish with  the assertion \ref{first assertion} to deduce that the limit as $t$ tends to infinty of the Fourier-Laplace transform $\mathbb{E}\left(e^{-\lambda Y_t+i\mu^{\mathsf{T}} X_t}\right)$ is equal to
\begin{multline*}
\underset{t\to\infty}{\lim}\displaystyle\int_0^{\infty}\int _{\R^n} \mathbb{A}_t(\mu,\theta,m,x_0) \mathbb{E}\left(e^{-\lambda Y_t-i\mu^{\mathsf{T}} U_t-\frac{1}{2}(\mu^{\mathsf{T}}\otimes\mu^{\mathsf{T}})V_t}\displaystyle\vert Y_0=y_0\right) \mathbb{P}_{(Y_0,X_0)}(\mathrm{d}y_0,\mathrm{d}x_0)\\=\exp\left(\tilde{g}_\infty\left(-\lambda,\mu,-\frac{1}{2}\mu \otimes\mu \right)+i\mu^{\mathsf{T}}\theta^{-1} m\right).	
\end{multline*}

The second assertion insures the application of Lévy's continuity theorem to prove the convergence $Z_t\stackrel{\mathcal{D}}{\longrightarrow}Z_{\infty}$ and to characterize the distribution of $Z_\infty=(Y_\infty,X_\infty)$ by its Fourier-Laplace transform as the limit function obtained above. We complete the proof of the first result in the case $\rho_{J1}=\mathbf{0}_n$ by considering the function $\R_+\times\mathcal{U}_1\times\R^n\ni (t,u_1,u_2)\mapsto K_t(u_1,u_2):=\tilde{K}_t(u_1,u_2,-\frac{1}{2}u_2\otimes u_2)$ and writing the associated Riccati equation from relation \eqref{K tilde EDO } given by
\begin{equation}
	\begin{cases}
		\begin{aligned}
			&\dfrac{\partial K_t}{\partial t}(u_1,u_2)=\frac{\rho_{11}^2}{2}K_t^2(u_1,u_2)-b K_t(u_1,u_2)-i\kappa^\mathsf{T}\,e^{-t\theta^\mathsf{T}}u_2 -\frac{1}{2}(\text{vec}(\rho_{JJ}\rho_{JJ}^{\mathsf{T}}))^\mathsf{T}\,e^{-t (\theta^{\mathsf{T}}\oplus\theta^{\mathsf{T}})}(u_2\otimes u_2),\\   &K_0(u_1,u_2)=u_1.
		\end{aligned}
	\end{cases}
	\label{EDK}
\end{equation} 
which is nothing but the relation \eqref{EDP_loi limite Stationarity theorem} for $\rho_{J1}=\mathbf{0}_n$. }

\item In order to prove the strict stationarity (translation invariance of the finite dimensional
distributions), we show that for all $t\in\R_+$, the distribution of $(Y_t,X_t)$ is translation invariant and has the same distribution of $(Y_{\infty},X_{\infty})$.
First, since according to $\eqref{(Y X) to (Y U V)}$ and $\eqref{Lap en K}$, we have
\begin{equation*}
	\mathbb{E}\left(e^{-\lambda Y_t+i\mu^{\mathsf{T}}X_t}\vert (Y_0,X_0)=(y_0,x_0)\right)=\exp\left(y_0 K_t(-\lambda,\mu)+i\mu^{\mathsf{T}}e^{-t \theta}x_0+g_t(-\lambda,\mu)\right),
\end{equation*}
where $g_t(-\lambda,\mu)=a\displaystyle\int_0^t K_s(-\lambda,\mu)\mathrm{~d}s+i\mu^{\mathsf{T}}\theta^{-1}(1-e^{-t \theta})m$, with $K$ solution of the Riccati equation \eqref{EDK}, from the first part of the theorem, it is enough to check that for all $t\in\R_{+}$, $\lambda\in\R_{+}$ and $\mu\in\R^n$, we have
\begin{equation*}
\mathbb{E}\left(\exp(K_t(-\lambda,\mu)Y_{\infty}+i\mu^{\mathsf{T}}e^{-t \theta}X_{\infty}+g_t(-\lambda,\mu)\right)=\exp\left(a\int_0^{\infty}K_s(-\lambda,\mu)\mathrm{~d}s+i\mu^{T}\theta^{-1}m\right).
\end{equation*}

As $K_t(-\lambda,\mu)\in\mathcal{U}_1$ for all $t\in\R_{+}$ and $(-\lambda,\mu)\in\mathcal{U}_1\times \R^n$, again from the first part of the theorem, the expectation above is equal to
\begin{multline}
\exp\left(a\int_0^{\infty}K_s(K_t(-\lambda,\mu),e^{-t\theta^{\mathsf{T}} }\mu)\mathrm{~d}s+i\mu^{\mathsf{T}}e^{-t \theta}\theta^{-1} m+g_t(-\lambda,\mu)\right)\\
=\exp\left(a\left(\int_0^{\infty}K_s(K_t(-\lambda,\mu),e^{-t\theta^{\mathsf{T}} }\mu)\mathrm{~d}s+\int_0^t K_s(-\lambda,\mu)\mathrm{~d}s\right)+i\mu^{\mathsf{T}}\theta^{-1}m\right).
\end{multline}
By simple comparaison of the two terms, it is enough to check that for all $(-\lambda,\mu)\in\mathcal{U}_1\times \R^n$
$
\displaystyle\int_t^{\infty}K_s(-\lambda,\mu)\mathrm{~d}s=\int_0^{\infty}K_s(K_t(-\lambda,\mu),e^{-t\theta^{\mathsf{T}} }\mu)\mathrm{~d}s,
$
which holds if
\begin{equation}
K_s(K_t(-\lambda,\mu),e^{-t\theta^{\mathsf{T}} }\mu)=K_{t+s}(-\lambda,\mu),
\label{equality}
\end{equation}
for all $s,t\in\R_{+}$ and $(-\lambda,\mu)\in\mathcal{U}_1\times \R^n$. Since, by the help of matrix properties given in relation \eqref{matrix property}, the functions $\R_+ \ni s\mapsto K_s(K_t(-\lambda,\mu),e^{-t\theta^{\mathsf{T}}}\mu)$ and $\R_+ \ni s \mapsto K_{t+s}(-\lambda,\mu)$, $s\in\R_+$ and $(-\lambda,\mu)\in\mathcal{U}_1\times \R^n$, satisfy the same differenitial equation $\eqref{EDK}$ with the same intial value $K_t(-\lambda,\mu)$, we conclude using the uniqueness of the solution.
Finally, using the fact that $(Z_t)_{t\in\R_{+}}$ is a time-homogeneous Markov process and the tower property, we deduce that the process $(Z_t)_{t\in\R_{+}}$ is strictly stationary which completes the second part of the proof. 
\end{enumerate}
\end{proof}
\begin{proof} (Proof of Lemma \ref{Stationarity_Lemma})
Let us consider the process 
\begin{equation}
(Y_t,\tilde{X}_t)^{\mathsf{T}}=A(Y_t,X_t)^{\mathsf{T}},\quad \text{with  } A=\begin{bmatrix}
	1& \mathbf{0}_{n}^{\mathsf{T}}\\
	-\frac{1}{\rho_{11}}\rho_{J1}& \mathbf{I}_n
\end{bmatrix}.
\label{Transformation}
\end{equation}
 It is easy to check that $(Y_t,\tilde{X}_t)^{\mathsf{T}}$ is an $\mathit{AD}(1,n)$ process with affine drift 
\begin{equation*}
	A \begin{bmatrix}
		a\\
		m
	\end{bmatrix}-A \begin{bmatrix}
		b&\mathbf{0}_{n}^{\mathsf{T}}\\
		\kappa&\theta
	\end{bmatrix}A^{-1}\begin{bmatrix}
		{Y}_t\\
		\tilde{X}_t
	\end{bmatrix}=\begin{bmatrix}
		a-b{Y}_t\\
		m-\frac{a}{\rho_{11}}\rho_{J1} -(\kappa-\frac{1}{\rho_{11}}(b\mathbf{I}_n-\theta)\rho_{J1}){Y}_t-\theta \tilde{X}_t 
	\end{bmatrix}
\end{equation*}
and a diffusion matrix
$\begin{bmatrix}
	\rho_{11}^2{Y}_t& \mathbf{0}_{n}^{\mathsf{T}}\\
	\mathbf{0}_n&{Y}_t \rho_{JJ} \rho_{JJ}^{\mathsf{T}} \
\end{bmatrix}
$. Then we can apply the results of Theorem \ref{Stationarity theorem}  on the process $(Y_t,\tilde{X}_t)^{\mathsf{T}}$ with the associated drift parameters $(a,b,\tilde{m},\tilde{k},\theta)$ and $\tilde{\rho}$, where $\tilde{m}=m-\frac{a}{\rho_{11}}\rho_{J1}\in\R^n$, $\tilde{\kappa}=\kappa-\frac{1}{\rho_{11}}(b\mathbf{I}_n-\theta)\rho_{J1}\in\R^n$ and $\tilde{\rho}=\begin{bmatrix}
	\rho_{11}& \mathbf{0}_{n}^{\mathsf{T}}\\
	\mathbf{0}_n&{Y}_t \rho_{JJ} \
\end{bmatrix}
$. Hence, the relation \eqref{EDP_loi limite Stationarity theorem} can be extended to
\begin{equation*}
	\begin{aligned}
		\mathbb{E}\left(e^{-\lambda Y_\infty +i \mu^{\mathsf{T}} X_\infty}\right)&=\mathbb{E}\left(e^{(-\lambda,i\mu)A^{-1} (Y_\infty,\tilde{X}_\infty)^{\mathsf{T}}}\right)=\mathbb{E}\left(e^{\left(-\lambda+i\frac{1}{\rho_{11}}\mu^{\mathsf{T}}\rho_{J1}\right)Y_\infty+i\mu^{\mathsf{T}}\tilde{X}_{\infty}}\right)\\&=\exp\left(a\displaystyle\int_0^{\infty} K_s\left(-\lambda+i\frac{1}{\rho_{11}}\mu^{\mathsf{T}}\rho_{J1},\mu \right)\mathrm{~d}s+i\mu^{\mathsf{T}}\theta^{-1}\left(m-\frac{a}{\rho_{11}}\rho_{J1} \right)\right)\\
		&=\exp\left(a\displaystyle\int_0^{\infty} \mathcal{K}_s\left(-\lambda,\mu \right)\mathrm{~d}s+i\mu^{\mathsf{T}}\theta^{-1}m\right),
	\end{aligned}
\end{equation*}
with $\mathcal{K}_t(u_1,u_2)=K_t\left(u_1+i\frac{1}{\rho_{11}}u_2^{\mathsf{T}}\rho_{J1},\mu\right)-i\dfrac{1}{\rho_{11}}u_2^{\mathsf{T}}e^{-t\theta}\rho_{J1}$, for all $(u_1,u_2)\in\mathcal{U}_1\times\R^n$, where $K_t$ is the function associated to the process $(Y_t,\tilde{X}_t)$ solution of the Riccati equation \eqref{EDP_loi limite Stationarity theorem} with  parameters $(a,b,\tilde{m},\tilde{\kappa},\theta)$ and $\tilde{\rho}$. This function satisfies $	\dfrac{\partial \mathcal{K}_t}{\partial t}(u_1,u_2)=\dfrac{\partial K_t}{\partial t}\left(u_1+i\frac{1}{\rho_{11}}u_2^{\mathsf{T}}\rho_{J1},u_2\right)+i\dfrac{1}{\rho_{11}}u_2^{\mathsf{T}}\theta e^{-t\theta}\rho_{J1}$. Hence, using the dynamic of $K_t$ and its relation with $\mathcal{K}_t$, it is easy to get $\dfrac{\partial \mathcal{K}_t}{\partial t}(u_1,u_2)$ is equal to
\begin{multline*}
\frac{\rho_{11}^2}{2}\left(\mathcal{K}_t(u_1,u_2)+i\dfrac{1}{\rho_{11}}\rho_{J1}^{\mathsf{T}}e^{-t\theta^{\mathsf{T}}}u_2\right)^2-b \left(\mathcal{K}_t(u_1,u_2)+i\dfrac{1}{\rho_{11}}\rho_{J1}^{\mathsf{T}}e^{-t\theta^{\mathsf{T}}}u_2\right)+i\dfrac{1}{\rho_{11}}\rho_{J1}^{\mathsf{T}} e^{-t\theta^{\mathsf{T}}}\theta^{\mathsf{T}}u_2\\ -i\left(\kappa^{\mathsf{T}}-\frac{1}{\rho_{11}}\rho_{J1}^{\mathsf{T}}(b\mathbf{I}_n-\theta^{\mathsf{T}})\right)\,e^{-t\theta^{\mathsf{T}}} u_2
-\frac{1}{2}\left(\text{vec}(\rho_{JJ}\rho_{JJ}^{\mathsf{T}})\right)^\mathsf{T}\,e^{-t (\theta^\mathsf{T}\oplus\theta^\mathsf{T})}(u_2\otimes u_2).
\end{multline*}
By a simple calculation, the last equation gives exactly the PDE \eqref{EDP_loi limite Stationarity theorem}.
Moreover, we have also the initial value $\mathcal{K}_0(u_1,u_2)=K_0\left(u_1+i\dfrac{1}{\rho_{11}}u_2^{\mathsf{T}}\rho_{J1},u_2\right)-i\dfrac{1}{\rho_{11}}u_2^{\mathsf{T}}\rho_{J1}=u_1$, this completes the extension of the first part of the theorem. For the second patrt, the strict stationarity of the process $(Y_t,X_t)_{t\in\R_+}$ is deduced by the one of $(Y_t,\tilde{X}_t)_{t\in\R_+}$
\end{proof}

\section{Exponential Ergodicity}
In the subcritical case, the following theorem states the exponential ergodicity for the process $(Z_t)_{t\in\R_+}$ defined by the SDE $\eqref{model01}$ and extends the results found by Barczy et al. in \cite{Barczy stationarity ergodicty} and Bolyog and Pap in \cite{Boylog1} for the special cases $n=1$. As a consequence, a strong law of large numbers given by the relation $\eqref{ergodic}$ is obtained by proposition 2.5 of Bhattacharya in \cite{Bhattacharaya}.
\begin{theoreme} Let us consider the $\mathit{AD}(1,n)$ model $\eqref{model01}$ with $a\in\R_{++}$, $b\in\R_{++}$, $m\in\R^n$, $\kappa\in\R^n$ and $\theta\in\mathcal{M}_n$ a diagonalizable positive definite matrix with initial random values $Z_0=(Y_0,X_0)^{\mathsf{T}}$ independent of $(B_t)_{t\in\R_{+}}$ satisfying $\mathbb{P}(Y_0\in\R_{++})=1$. Then the process $Z$ is exponentially ergodic, namely, there exists $\delta\in\R_{++},\ B\in\R_{++}$ and $r\in\R_{++}$ such that
\begin{equation}
\underset{\vert g\vert\leq V+1}{\sup}\left\vert \mathbb{E}\left(g(Z_t)\vert Z_0=z_0\right)-\mathbb{E}(g(Z_{\infty}))\right\vert\leq B(V(z_0)+1))e^{-\delta t},
\label{ergodicity1}
\end{equation}
for all $t\in\R_{+}$ and $z_0=(y_0,x_0)^{\mathsf{T}}\in\mathcal{D}$, where the supremum is running for Borel measurable functions $g:\mathcal{D}\to\R$, $V(y,x):=y^2+r\norm{x}_2^2$, for all $(y,x)\in\mathcal{D}$,
and $Z_\infty=(Y_\infty,X_\infty)^{\mathsf{T}}$ is defined by $\eqref{loi limite Stationarity theorem}$. Moreover, for all Borel measurable functions $f:\R\times\R^n \to \R$ such that $\mathbb{E}\left(\vert f(Z_\infty) \vert\right)<\infty$, we have
\begin{equation}
\mathbb{P}\left( \underset{T\to\infty}{\lim}\dfrac{1}{T}\displaystyle\int_0^Tf(Z_s)\mathrm{~d}s=\mathbb{E}\left(f(Z_\infty)\right)\right)=1.
\label{ergodic}
\end{equation}
\label{ergodicity theorem}
\end{theoreme}
\begin{proof}
At first, we assume that $\rho_{J1}=\mathbf{0}_{n}$.
In order to prove the exponential ergodicity given by $\eqref{ergodicity1}$, we use the Foster-Lyapunov criteria, see \cite[Theorem 6.1]{Foster-Lyapunov}. Hence, it is enough to check the three following assertions:
\begin{enumerate}[label=(\roman*)]
\item $(Z_t)_{t\in\R_+}$ is a Borel right process (defined as in Getoor \cite[page 55]{Getoor} and Sharpe \cite[page 38]{Sharpe}).
\item For the skeleton chain $(Z_k)_{k\in\N}$ all compact sets are petite (skeleton chains and petite sets are defined as in Meyn and Tweedie \cite[pages 491 and 500]{Meyn II} and \cite[pages 550]{Meyn I}, respectively).
\item There exists $c\in\R_{++}$ and $d\in\R$ such that the inequality
\begin{equation*}
(\mathcal{A}_k\, V)(z)\leq -c\,V(z)+d,\quad z\in\mathcal{O}_k,
\end{equation*}
holds for all $k\in\N$, where $\mathcal{O}_k:=\lbrace z\in\mathcal{D}:\norm{z}_2<k\rbrace$, for each $k\in\N$, $\mathcal{A}_k$ denotes the extended generator of the process $(Z_t^{(k)})_{t\in\R_+}$ given by
\begin{equation*}
Z_t^{(k)}:=\begin{cases}
Z_t,\quad&t<T_k,\\
(0,k,\ldots,k)^{\mathsf{T}},\quad&t\geq T_k,
\end{cases}
\end{equation*}
 for $t\in\R_+$, where the stopping time $T_k$ is defined by $T_k:=\inf\lbrace t\in \R_+ : Z_t \in \mathcal{D}\backslash \mathcal{O}_k\rbrace$. Note that we can choose instead of $(0,k,\ldots,k)^{\mathsf{T}}$, any other point in $\mathcal{D}\setminus\mathcal{O}_k$. 
\end{enumerate}
In order to prove the first assertion, thanks to Meyn and Tweedie \cite[page 498]{Meyn II},  it is enough to check that $(Z_t)_{t\in\R_{+}}$ is a weak Feller process with continuous sample paths satisfying strong Markov property (the Feller and weak Feller properties are given in Meyn and tweedie \cite[section 3.1]{Meyn II}). By Duffie et al. \cite[proposition 8.2 or Theorem 2.7]{Duffie}, the $\mathit{AD}(1,n)$ model \eqref{model01}, as an affine process, is a Feller Markov
process. Besides, since $(Z_t)_{t\in\R_{+}}$ has continuous sample paths almost surely, it is automatically a strong Markov process (see, e.g., Chung \cite[Theorem 1, page 56]{Chung}). For the second assertion, since $(Z_k)_{k\in\N}$ is Feller, it is sufficient, by Proposition 6.2.8 in Meyn and Tweedy \cite{Meync'}, to check that it is irreducible. Note that the irreducibilty of $(Z_k)_{k\in\N}$ holds if  the conditional distribution of $Z_1=(Y_1,X_1)^{\mathsf{T}}$ given $Z_0=(Y_0,X_0)^{\mathsf{T}}$ is absolutely continuous with respect to the Lebesgue
measure on $\mathcal{D}$ such that the conditional density function $f_{Z_1\vert Z_0}$ is positive on $\mathcal{D}$. Indeed, the Lebesgue measure on $\mathcal{D}$ is $\sigma$-finite and
if $B$ is a Borel set in $\mathcal{D}$ with positive Lebesgue measure, then
\begin{equation*}
\begin{aligned}
\mathbb{E}\left(\displaystyle\sum_{k=0}^{\infty}\mathds{1}_B(Z_k)\displaystyle\vert Z_0=z_0\right)\geq\mathbb{P}(Z_1\in B\vert Z_0=z_0)=\displaystyle\int_B f_{Z_1\vert Z_0}(z\vert z_0)\mathrm{~d}z,
\end{aligned}
\end{equation*}
for all $z_0\in\mathcal{D}$. The existence of this conditional density with the required property can be checked as follows. By taking $t=1$ in the expression $\eqref{X expression}$, 
we consider the vector $({Y}_1,\check{X}_1)^{\mathsf{T}}$, where $\check{X}_1=(\check{X}_1^1,\ldots,\check{X}_1^n)^{\mathsf{T}}$ is the random part of $X_1$ given by  
$\check{X}_1:= -\displaystyle\int_0^1  Y_s e^{s \theta}\kappa \mathrm{~d}s +\displaystyle\int_0^1 \sqrt{Y_s} e^{s \theta} \rho_{JJ}\mathrm{~d}W_s$.
Since $Y_1\geq 0$, it is sufficient to check that, for all $x\in\R^n$, $({Y}_1,\check{X}_1)^{\mathsf{T}}$ has a positive density function on $\mathcal{D}$.
Note that the conditional distribution of $\check{X}_1$ given $(Y_t)_{t\in[0,1]}$ is a normal distribution with mean $m_Y=-\displaystyle\int_0^1 Y_se^{s \theta} \kappa \mathrm{~d}s$ and covariance matrix $C_Y=\displaystyle\int_0^1 Y_s e^{s \theta} \rho_{JJ} \rho_{JJ}^{\mathsf{T}} e^{s\theta^{\mathsf{T}}}\mathrm{~d}s$. In order to compute the distribution function of $({Y}_1,\check{X}_1)^{\mathsf{T}}$, if we set $p(u)$ this conditional density function, then for all $(y,x)\in\mathcal{D}$, we have
\begin{equation*}
\begin{aligned}
\mathbb{P}\left(Y_1\leq y,\,\check{X}_1^{1}\leq x_1,\,\ldots,\,\check{X}_1^{n}\leq x_n \right)&=\mathbb{E}\left(\mathds{1}_{\lbrace Y_1\leq y\rbrace}\,\mathbb{P}\left(\check{X}_1^{1}\leq x_1,\,\ldots,\,\check{X}_1^{n}\leq x_n\mid (Y_t)_{t\in[0,1]}\right)\right)\\
&=\mathbb{E}\left( \mathds{1}_{\lbrace Y_1\leq y\rbrace}\displaystyle\int_{\prod_{i=1}^n(-\infty,x_i)} p(u)\mathrm{~d}u\right),
\end{aligned}
\end{equation*}
where $p(u)= \dfrac{(2\pi)^{-n/2}}{\sqrt{\det(C_Y)}} \exp\left(-\dfrac{1}{2}(u-m_Y)^\mathsf{T}C_Y^{-1}(u-m_Y) \right)$, for all $u\in\R^n$. By conditioning on $Y_1$ and using Fubini-Tonelli property, we get
\begin{equation}
\mathbb{P}\left(Y_1\leq y,\,\check{X}_1^{1}\leq x_1,\,\ldots,\,\check{X}_1^{n}\leq x_n \right)= \displaystyle\int_0^y \int_{\prod_{i=1}^n(-\infty,x_i)} \mathbb{E}\left(p(u) \vert Y_1=z\right)f_{Y_1}(z)\mathrm{~d}z\mathrm{~d}u,
\label{density Y1 check X1}
\end{equation}
where $f_{Y_1}$ denotes the density function of $Y_1$ given that $Y_0=y_0\in\R_{++}$  defined, for all $y\in\R$, by
$$f_{Y_1}(y)=\dfrac{2be^{b(2a+1)}}{e^b-1}\left(\dfrac{y}{y_0}\right)^{a-\frac{1}{2}}\exp\left(\dfrac{-2b(y_0+e^by)}{e^b-1}\right)I_{2a-1}\left(\dfrac{2b\sqrt{y_0 y}}{\sinh(\frac{b}{2})}\right)\mathds{1}_{y\in\R_{++}},$$
where $I_{2a-1}(x)=\displaystyle\sum_{m=0}^{\infty}\dfrac{1}{\fact{m} \Gamma(m+2a)}\left(\dfrac{x}{2}\right)^{2m+2a-1}$, for all $x\in\R_{++}$. As $f_{Y_1}$ is a positive density (see, e.g., \cite{Jeanblanc}) and $p(u)$ is finite and positive, the random vector $(Y_1,\check{X}_1)^{\mathsf{T}}$ has a positive density on $\mathcal{D}$. Concerning the last assertion, note that the extended generator,  according to \cite[Theorem 1.1]{Friesen}, is given by
\begin{equation*}
(\mathcal{A}_k V)(y,x)=y+nry+2(ay-by^2)+ 2rx^{\mathsf{T}}(m-\kappa y-\theta x),
\end{equation*}
for all $(y,x)^{\mathsf{T}}\in \mathcal{O}_k,\, k\in\N.$
Therefore, by simple calculation, we get
\begin{equation*}
\begin{aligned}
(\mathcal{A}_k V)(y,x)+cV(y,x)
\leq-c_2\left(y+\frac{r}{c_2}\kappa^{\mathsf{T}}x\right)^2-x^{\mathsf{T}}c_3\, x+c_1\left(y+\frac{r}{c_2}\kappa^{\mathsf{T}}x\right)+c_4 x,
\end{aligned}
\end{equation*}
 with $c_1=1+nr+2a,\ $  $c_2=2b-c,\ $  $c_3=r(2\lambda_{\min}(\theta)-c)\mathbf{I}_n-\dfrac{r^2}{c_2}\kappa\kappa^{\mathsf{T}}$ and  $c_4=2rm^{\mathsf{T}}-r\dfrac{c_1}{c_2}\kappa^{\mathsf{T}}$. Next, in order to get $c_2\in\R_{++}$ and $c_3$ a positive definite matrix, we choose $
0<c<2(\lambda_{\min}(\theta)\wedge b)$ and  $0<r<\dfrac{(2\lambda_{\min}(\theta)-c)(2b-c)}{\lambda_{\max}(\kappa\kappa^{\mathsf{T}})}$.
Furthermore, since $c_3$ is symmetric, then $c_3=\tilde{c}_3\tilde{c}_3^{\mathsf{T}}$, with $\tilde{c}_3$ is positive definite matrix. Hence, we can bounded $(\mathcal{A}_k V)(y,x)+cV(y,x)$ by
\begin{equation*}
-c_2\left(y+\frac{r}{c_2}\kappa^{\mathsf{T}}x-\dfrac{c_1}{2c_2}\right)^2-\left(x^{\mathsf{T}}\tilde{c}_3-\frac{1}{2}c_4({\tilde{c}_3^{-1}})^{\mathsf{T}}\right)\left(x^{\mathsf{T}}\tilde{c}_3-\frac{1}{2}c_4({\tilde{c}_3^{-1}})^{\mathsf{T}}\right)^{\mathsf{T}}+d,
\end{equation*}
with $d=\dfrac{c_1^2}{4c_2}+\dfrac{1}{4}c_4c_3^{-1}c_4^{\mathsf{T}}
$. This completes the proof of the exponential ergodicity theorem in the special case $\rho_{J1}=\mathbf{0}_n$. In the general case, we use the transformation \eqref{Transformation} introduced in the proof of lemma \ref{Stationarity_Lemma}. Let $g$ be a Borel measurable function satisfying $\lvert g(y,x)\rvert \leq V(y,x)+1$, for all $(y,x)^{\mathsf{T}}\in\mathcal{D}$, then by putting $h:=g\circ A^{-1}$, it is easy to check that, for all $(y,\tilde{x})^{\mathsf{T}}\in\mathcal{D}$, we have
\begin{align*}
	\lvert h(y,\tilde{x})\rvert\leq V(A^{-1}(y,x))=y^2+r\norm{\dfrac{y}{\rho_{11}}\rho_{J1}+\tilde{x}}_2^2+1\leq C(V(y,\tilde{x})+1),
\end{align*}  
with $C= \left(1+2r\dfrac{\norm{\rho_{J1}}_2^2}{\rho_{11}^2}\right)\vee 2$. 
 Consequently, using he inequality $\eqref{ergodicity1}$ for the process $(Y_t,\tilde{X}_t)_{t\in\R_{+}}$ with the Borel measurable function $\dfrac{1}{C}h$, we get 
\begin{align*}
\left\vert \mathbb{E}\left(g(Y_t,X_t)\vert (Y_0,X_0)=(y_0,x_0)\right)-\mathbb{E}(g(Y_{\infty},X_{\infty}))\right\vert&\leq BC \left(y_0^2+r\norm{x_0-\dfrac{y_0}{\rho_{11}}\rho_{J1}}_2^2+1\right)e^{-\delta t}\\
&\leq BC^2 (V(y_0,x_0)+1)e^{-\delta t}.
\end{align*}
This completes the proof by changing $BC^2$ to $B$.
\end{proof}
The main goal of the next section is to study the maximum likelihood estimator of the drift parameters in the subcritical and a special supercritical cases using the results obtained in the previous sections.
\section{Maximum likelihood estimation}
In this section, let us recall that for statistical estimations with continuous observations, we always suppose that the diffusion parameter ${\rho}$ is known, see e.g. \cite[page 50]{Kutoyants2004}. In fact, using an arbitrarily short continuous time observation of the process  $(Z_t)_{t\in[0,T]}$, for some $T\in\R_{++}$, the diffusion coefficient is $\sigma(Z_t,t\in[0,T])$-measurable. In fact, we have $\rho\rho^{\mathsf{T}}=\left(\int_0^T Y_s \mathrm{d}s\right)^{-1}\langle Z\rangle_T$ and for all $i,j\in\lbrace1,\ldots,d\rbrace$, the bracket $(\langle Z \rangle_T)_{i,j}=\langle Z^i ,Z^j\rangle_T$ can be easily approximated by $\sum_{\ell=1}^{p_m} (Z^i_{t_{\ell}^m}-Z^i_{t_{\ell-1}^{m}})(Z^j_{t_{\ell}^m}-Z^j_{t_{\ell-1}^{m}})$, where $t_0^m=0\leq t_{1}^m \leq \cdots \leq t_{p_m}^m=T$ is a sequence of subdivisons of $[0,T]$ satisfying
 $\underset{1\leq \ell \leq p_m}{\sup} \abs{t_{\ell}^m - t_{\ell-1}^m} \underset{m\to\infty}{\longrightarrow}0$. Note that this convergence holds in probability and as $\rho\rho^{\mathsf{T}}$ is symmetric and positive definite, the exact expression of $\rho$ can be obtained using the Cholesky decomposition.
 
\subsection{Existence and uniqueness of MLE}
Let $T\in\R_{++}$. In this subsection, we formulate a proposition about existence and uniqueness of the MLE $\tau_T^{\star}$ of $\tau$ based on the continuous observations $(Z_t)_{t\in[0,T]}$ using Lipster and Shiryaev \cite{Lipster}.
\begin{proposition}
	Let $a\in\R_{++}$, $b\in\R$, $m,\kappa\in\R^n$ and $\theta\in\mathcal{M}_n$. Let $(Z_t)_{t\in\R_+}$ be the unique strong solution of the SDE $\eqref{model0 Z}$ with initial random values $Z_0=(Y_0,X_0)^{\mathsf{T}}$ independent of $(B_t)_{t\in\R_{+}}$ satisfying $\mathbb{P}(Y_0\in\R_{++})=1$. Then, for each $T\in\R_{++}$ there exists a unique MLE of $\tau$ almost surely having the form
	\begin{equation}
		\tau_T^{\star}= \left(\displaystyle\int_0^T \Lambda(Z_s)^\mathsf{T} S(Z_s)^{-1} \Lambda(Z_s)\mathrm{~d}s\right)^{-1}\displaystyle\int_0^T  \Lambda(Z_s)^\mathsf{T} S(Z_s)^{-1} \mathrm{~d}Z_s,
		\label{MLE expression}
	\end{equation}
where, for all $s\in[0,T]$, $S(Z_s)=Z_s^1 \rho\rho^{\mathsf{T}}$.
\label{MLE estimator}
\end{proposition}
\begin{proof}
Let $(\tilde{Z}_t)_{t\in\R_+}$ solution to $\mathrm{d}\tilde{Z}_t=\sqrt{\tilde{Z}_t^1}\,\rho\mathrm{~d}B_t$, for all $t\in\R_{+}$, with initial random value $\tilde{Z}_0$ satisfying $\mathbb{P}(\tilde{Z}_0=Z_0)=1$. Hence,
using the relation (7.138) in Lipster and Shiryaev \cite[page 297]{Lipster} applied on the processes $(Z_t)_{t\in[0,T]}$ and $(\tilde{Z}_t)_{t\in[0,T]}$, we deduce the associated likelihood ratio $L\left(\tau,(Z_t)_{t\in[0,T]}\right)=\dfrac{\mathrm{d}\mathbb{P}_Z}{\mathrm{d}\mathbb{P}_{\tilde{Z}}}((Z_t)_{t\in[0,T]})$. Taking its logarithm, we get the log-likelihood ratio given by
\begin{equation}
\ln L\left(\tau,(Z_t)_{t\in[0,T]}\right)=\displaystyle\int_0^T \tau^\mathsf{T} \Lambda(Z_s)^\mathsf{T} S(Z_s)^{-1} \mathrm{~d}Z_s -\dfrac{1}{2} \displaystyle\int_0^T \tau^\mathsf{T}\Lambda(Z_s)^\mathsf{T} S(Z_s)^{-1} \Lambda(Z_s)\tau\mathrm{~d}s.
\label{Log-Likelihood}
\end{equation}
Since the MLE is the vector $\tau_T^{\star}$ that maximizes the likelihood $L$, we compute the gradient of $\ln L$,
$$\triangledown_{\tau}\ln L(\tau,(Z_t)_{t\in[0,T]})=\displaystyle\int_0^T  \Lambda(Z_s)^\mathsf{T} S(Z_s)^{-1} \mathrm{~d}Z_s - \displaystyle\int_0^T \Lambda(Z_s)^\mathsf{T} S(Z_s)^{-1} \Lambda(Z_s)\mathrm{~d}s\,\tau$$
and its Hessian matrix 
$$H_{\tau}\ln L(\tau,(Z_t)_{t\in[0,T]})=-\displaystyle\int_0^T \Lambda(Z_s)^\mathsf{T} S(Z_s)^{-1} \Lambda(Z_s)\mathrm{~d}s.
$$
As the later matrix is negative definite, we have on one hand, the MLE of $\tau$ is the unique zero of $\triangledown_{\tau}\ln L(\tau,(Z_t)_{t\in[0,T]})$ and on the other hand, we deduce that the matrix $\displaystyle\int_0^T \Lambda(Z_s)^\mathsf{T} S(Z_s)^{-1} \Lambda(Z_s)\mathrm{~d}s$ is invertible which completes the proof.
\end{proof}
\subsection{Consistency of the MLE: subcritical case}
\begin{proposition}
Let $a> \frac{\sigma_1^2}{2}$, $b\in\R_{++}$, $m,\kappa\in\R^n$ and $\theta\in\mathcal{M}_n$ a diagonalizable positive definite matrix. Let $(Z_t)_{t\in\R_+}$ be the unique strong solution of the SDE $\eqref{model0 Z}$ with initial random values $Z_0=(Y_0,X_0)^\mathsf{T}$ independent of $(B_t)_{t\in\R_{+}}$ satisfying $\mathbb{P}(Y_0\in\R_{++})=1$. Then the MLE $\tau_T^{\star}$ of $\tau$ given by the relation $\eqref{MLE expression}$ is strongly consistent, i.e. 
$\mathbb{P}\left(\underset{T\to\infty}{\lim} \tau_T^{\star}=\tau\right)=1.$ 
\label{consistency Theorem}
\end{proposition}
\begin{proof}
By the help of the SDE $\eqref{model0 Z}$ associated to the process $(Z_t)_{t\in[0,T]}$, we deduce that the drift parameter vector is written as follows
\begin{equation*}
\tau=\left(\displaystyle\int_0^T \Lambda(Z_s)^\mathsf{T} S^{-1}(Z_s)\Lambda(Z_s)\mathrm{~d}s\right)^{-1}\left(\displaystyle\int_0^T \Lambda(Z_s)^\mathsf{T} S^{-1}(Z_s)\mathrm{~d}Z_s-\displaystyle\int_0^T \sqrt{Z_s^1}\Lambda(Z_s)^\mathsf{T} S^{-1}(Z_s)\rho\mathrm{~d}B_s\right).
\end{equation*}
Hence, the error term $\tau_T^{\star}-\tau$ has the form $\langle M \rangle_T^{-1} M_T$,
where $(M_t)_{t\in[0,T]}$ is a Brownian martingale defined, for all $t\in[0,T]$, by
$
M_t=\displaystyle\int_0^t\dfrac{1}{\sqrt{Z_s^1}}\left(\rho^{-1}\Lambda(Z_s)\right)^{\mathsf{T}}\mathrm{~d}B_s,
$
with the quadratic variation
$
\langle M\rangle_t =\displaystyle\int_0^t \dfrac{1}{Z_s^1}\left(\rho^{-1}\Lambda(Z_s)\right)^{\mathsf{T}}\rho^{-1}\Lambda(Z_s)\mathrm{~d}s
$. By simple matrix calculation tools, we are able to explicit the expression of each component of $M_T=(M_T^{1},\ldots,M_T^{d^2+1})^{\mathsf{T}}$, thus we obtain 
$$M_T^{1}=\displaystyle\int_0^T \dfrac{1}{\sqrt{Z_s^1}}\displaystyle\sum_{i=1}^{d}(\rho^{-1})_{i,1}\mathrm{~d}B_s^i,\quad M_T^{2}=-\displaystyle\int_0^T {\sqrt{Z_s^1}}\displaystyle\sum_{i=1}^{d}(\rho^{-1})_{i,1}\mathrm{~d}B_s^i$$
and for all $k\in\lbrace 0,\ldots d-2\rbrace$ and $j\in\lbrace 0,\ldots d\rbrace$, we get
$M_T^{3+j+k(d+1)}=-\displaystyle\int_0^T \dfrac{1}{\sqrt{Z_s^1}}Z_s^j\displaystyle\sum_{i=k+2}^{d}(\rho^{-1})_{i,k+2}\mathrm{~d}B_s^i,$
with taking $Z_s^0=-1$. We can write the error $$\tau_T^{\star}-\tau=\left(D_T\langle M \rangle_T\right)^{-1} \left(D_TM_T\right),$$ where $D_T$ is a diagonal matrix containing, up to a constant, the inverses of the brackets of $M_T^i$, $i\in\lbrace1,\ldots,d^2+1\rbrace$, namely, $D_T=\text{diag}(V_T)$ where $V_T=(V_T^{1},\ldots,V_T^{d^2+1})^{\mathsf{T}}$ is the random vector defined as follows 
$$V_T^{1}= \left({\displaystyle\int_0^T \dfrac{1}{Z_s^1}\mathrm{~d}s}\right)^{-1},\quad V_T^{2}= \left({\displaystyle\int_0^T Z_s^1\mathrm{~d}s}\right)^{-1}$$
and for all $k\in\lbrace 0,\ldots d-2 \rbrace$ and $i\in\lbrace 0,\ldots d\rbrace$,
$V_T^{3+i+k(d+1)}= \left({\displaystyle\int_0^T \dfrac{(Z_s^i)^2}{Z_s^1}\mathrm{~d}s}\right)^{-1}$. Now at first, as $a>\dfrac{\sigma_1^2}{2}$, we have $\mathbb{E}\left(\dfrac{1}{Z_{\infty}^1}\right)=\dfrac{2b}{2a-\sigma_1^2}<\infty$, $\mathbb{E}\left({Z_{\infty}^1}\right)=\dfrac{a}{b}<\infty$ and $\mathbb{E}\left(\dfrac{(Z_{\infty}^i)^2}{Z_{\infty}^1}\right)\leq \mathbb{E}\left((Z_{\infty}^i)^{2\alpha}\right)^{\frac{1}{\alpha}}\mathbb{E}\left(\dfrac{1}{(Z_{\infty}^1)^\beta}\right)^{\frac{1}{\beta}}$ which is finite
for all $\alpha\in\R_{++}$ and $\beta\in]0,2a[$ such that $\frac{1}{\alpha}+\frac{1}{\beta}=1$, thanks to \cite[Proposition 1]{Alaya} for the negative moment of $Z_{\infty}^1$ and to the inequality of Burkholder Davis Gundy applied on the martingale part of the expression \eqref{X expression} of $X_t$. Secondly, we deduce using Theorem \ref{ergodicity theorem} that
\begin{equation}
	\label{D_T convergence}
	\begin{split}
\dfrac{1}{T}\displaystyle\int_0^T \dfrac{1}{Z_s^1}\mathrm{~d}s&\stackrel{a.s.}{\longrightarrow}\dfrac{2b}{2a-\sigma_1^2},\quad \text{as }T\to\infty,\\
	\dfrac{1}{T}\displaystyle\int_0^T{Z_s^1}\mathrm{~d}s&\stackrel{a.s.}{\longrightarrow}\dfrac{a}{b},\quad \text{as }T\to\infty,\\
\dfrac{1}{T}\displaystyle\int_0^T \dfrac{(Z_s^i)^2}{Z_s^1}\mathrm{~d}s&\stackrel{a.s.}{\longrightarrow}\mathbb{E}\left(\dfrac{(Z_{\infty}^i)^2}{Z_{\infty}^1}\right),\quad \text{as }T\to\infty,\text{  for all }i\in\lbrace2,\ldots, d\rbrace.
	\end{split}
\end{equation}
 Consequently, by Theorem \ref{Lipster Shiryaev theorem} in the appendix, we get $D_T M_T \stackrel{a.s.}{\longrightarrow} \mathbf{0}_{d^2+1}$ as $T$ tends to infinity. Hence it remains to prove that $D_T\langle M\rangle_T=\left(T D_T\right)\left(\dfrac{1}{T}\langle M\rangle_T\right)$ converges almost surely to an invertible limit matrix, as $T$ tends to infinity. On one hand, we have $T D_T$ converges by the relation \eqref{D_T convergence} and its limit is invertible, on the other hand, up to a constant, the components of $\langle M\rangle_T$ have one of the following terms: $T$, $\displaystyle\int_0^T \dfrac{1}{Z_s^1}\mathrm{~d}s$, $\displaystyle\int_0^T{Z_s^1}\mathrm{~d}s$, $\displaystyle\int_0^T \dfrac{(Z_s^i)^2}{Z_s^1}\mathrm{~d}s$, $i\in\lbrace 2,\ldots,d\rbrace$, $\displaystyle\int_0^T \dfrac{Z_s^i}{Z_s^1}\mathrm{~d}s$, $i\in\lbrace 2,\ldots,d\rbrace$ and $\displaystyle\int_0^T \dfrac{Z_s^iZ_s^j}{Z_s^1}\mathrm{~d}s$, $i, j\in\lbrace 2,\ldots,d\rbrace$ with $i\neq j$. Hence, similarly as above, since $\mathbb{E}\left(\dfrac{Z_{\infty}^i}{Z_{\infty}^1}\right)\leq \mathbb{E}\left((Z_{\infty}^i)^{\alpha}\right)^{\frac{1}{\alpha}}\mathbb{E}\left(\dfrac{1}{(Z_{\infty}^1)^\beta}\right)^{\frac{1}{\beta}}$ and $\mathbb{E}\left(\dfrac{Z_{\infty}^iZ_{\infty}^j}{Z_{\infty}^1}\right)\leq \mathbb{E}\left((Z_{\infty}^i)^{\alpha}\right)^{\frac{1}{\alpha}}\mathbb{E}\left((Z_{\infty}^j)^{\gamma}\right)^{\frac{1}{\gamma}}\mathbb{E}\left(\dfrac{1}{(Z_{\infty}^1)^\beta}\right)^{\frac{1}{\beta}}$ which are finite
 for all $\alpha,\gamma\in\R_{++}$ and $\beta\in]0,2a[$ such that $\frac{1}{\alpha}+\frac{1}{\gamma}+\frac{1}{\beta}=1$, we deduce by Theorem \ref{ergodicity theorem} that  
\begin{equation}\label{convergence  bracket}
	\begin{split}
		 \dfrac{1}{T}\displaystyle\int_0^T \dfrac{Z_s^i}{Z_s^1}\mathrm{~d}s&\stackrel{a.s.}{\longrightarrow}\mathbb{E}\left(\dfrac{Z_{\infty}^i}{Z_{\infty}^1}\right),\quad \text{as }T\to\infty,\text{  for all }i\in\lbrace2,\ldots, d\rbrace,\\
		\dfrac{1}{T}\displaystyle\int_0^T \dfrac{Z_s^iZ_s^j}{Z_s^1}\mathrm{~d}s&\stackrel{a.s.}{\longrightarrow}\mathbb{E}\left(\dfrac{Z_{\infty}^iZ_{\infty}^j}{Z_{\infty}^1}\right),\quad \text{as }T\to\infty,\text{  for all }i\neq j \in\lbrace2,\ldots, d\rbrace.
	\end{split}
\end{equation}
 By relations \eqref{D_T convergence} and \eqref{convergence  bracket}, we conclude that $$\dfrac{1}{T}\langle M\rangle_T\stackrel{a.s.}{\longrightarrow}\mathbb{E}\left(\dfrac{1}{Z_\infty^1}\Lambda(Z_\infty)^\mathsf{T}\left({\rho^{-1}}\right)^{\mathsf{T}}\rho^{-1}\Lambda(Z_\infty)\right),\quad \text{as }T\to\infty.$$  Next, we need to prove that this limit matrix is invertible. To do that, it is sufficient to prove that it is positive definite, so let $y\in\R^{d^2+1}$ a non-null vector and consider $\mathbb{E}\left(\dfrac{1}{Z_\infty^1}y^{\mathsf{T}}\Lambda(Z_\infty)^\mathsf{T}\left({\rho^{-1}}\right)^{\mathsf{T}}\rho^{-1}\Lambda(Z_\infty)y\right)$. Since, in one hand, the vector $\Lambda(Z_\infty)y$ is a combination of $(1,Z_\infty^1,\ldots,Z_\infty^d)$ and on the other hand, $(Z_\infty^1,\ldots,Z_\infty^d)^{\mathsf{T}}$ has a density thanks to the strict stationarity of $(Z_t)_{t\in\R_{+}}$ and to relation \eqref{density Y1 check X1}, then almost surely it is different to zero. We complete the proof of Theorem \ref{consistency Theorem} using the positive definite property of the matrix $\left({\rho^{-1}}\right)^{\mathsf{T}}\rho^{-1}$.  
\end{proof}

\subsection{Asymptotic behavior of the MLE: subcritical case}

In the sequel, in order to study the asymptotic behavior of the MLE $\tau_T^{\star}$ of $\tau$, we will use the central limit theorem (CLT) for martingales, see Theorem \eqref{CLT Van Zanten} in the appendix. 
\begin{theoreme}
Let $a> \dfrac{\sigma_1^2}{2}$, $b\in\R_{++}$, $m,\kappa\in\R^n$ and $\theta\in\mathcal{M}_n$ a diagonalizable positive definite matrix. Let $(Z_t)_{t\in\R_+}$ be the unique strong solution of the SDE $\eqref{model0 Z}$ with initial random values $Z_0=(Y_0,X_0)^\mathsf{T}$ independent of $(B_t)_{t\in\R_{+}}$ satisfying $\mathbb{P}(Y_0\in\R_{++})=1$. Then the MLE $\tau_T^{\star}$ of $\tau$ given by $\eqref{MLE expression}$ is asymptotically normal, namely
\begin{equation*}
\sqrt{T}(\tau_T^{\star}-\tau)\stackrel{\mathcal{D}}{\longrightarrow}\mathcal{N}(0,\mathcal{V}),\quad \text{as }T\to\infty,
\end{equation*}
where $\mathcal{V}$ is the inverse matrix of $\mathbb{E}\left(\dfrac{1}{Z_\infty^1}\Lambda(Z_\infty)^\mathsf{T}\left({\rho^{-1}}\right)^{\mathsf{T}}\rho^{-1}\Lambda(Z_\infty)\right)$.
\label{Theorem asymptotic behavior subcritical case}
\end{theoreme}
\begin{proof}
By the martingale presentation of the error term introduced in the proof of Proposition \ref{consistency Theorem} above we have $$\sqrt{T}(\tau_T^{\star}-\tau)= \left(\frac{1}{T}\langle M\rangle_T\right)^{-1}\frac{1}{\sqrt{T}}M_T.$$ Furthermore, we have established that 
\begin{equation}
\left(\frac{1}{T}\langle M\rangle_T\right)^{-1}\stackrel{a.s.}{\longrightarrow}  \mathcal{V},\quad \text{as }T\to\infty.
\end{equation}
Hence, by the central limit theorem for martingales (see Theorem \ref{CLT Van Zanten} in the appendix), we deduce that
\begin{equation*}
\frac{1}{\sqrt{T}}M_T \stackrel{\mathcal{D}}{\longrightarrow}\mathcal{N}(0,\mathcal{V}^{-1}),\quad \text{as }T\to\infty.
\end{equation*}
This completes the proof.
\end{proof}
In the following subsection we will treat one subclass of the supercritical case for which the process associated to the $\mathit{AD}(1,n)$ model \ref{model0 Z} is non-ergodic and we will call it "a special supercritical case".
\subsection{A special supercritical case}
In this part, we consider the $d^2-n$-dimensional vector $\tilde{\tau}=\left(b,\kappa_1,\theta_{11},\ldots,\theta_{1n},\ldots,\kappa_n,\theta_{n1},\ldots,\theta_{nn}\right)^{\mathsf{T}}$ as the unknown drift parameter vector. Next, we formulate a proposition about the unique existence of the MLE $\tilde{\tau}$ of $\tau$ based on the continuous observations $(Z_t)_{t\in[0,T]}$, for $T\in\R_{++}$, in a special supercritical case.
\begin{remarque}
	In the supercritical case, the MLE of $a$ and $m$ are not even weakly consistent due to the presence of the integral $V_T^{1}$ defined in the proof of Proposition \ref{consistency Theorem} in its error terms (see \cite[Proposition 4]{Alaya1}). For that reason, in what follows the parameter $c=(a,m)^{\mathsf{T}}$ is supposed to be known. However, the MLE of $b$ is strongly consistent and, for  all $i,j\in\lbrace 1,\ldots,n\rbrace$, the MLE of $\kappa_i$ and $\theta_{ij}$ are weakly consistent, see Theorem \ref{normality theorem supercritical case}. In the remaining cases, in a similar way by assuming that all non-consistent drift parameters to be known, we are able to study statistical estimations related to the $\mathit{AD}(1,n)$ model.  
\end{remarque}
Next, we present a second version of Proposition \ref{MLE estimator}, when $c=(a,m)^{\mathsf{T}}$ is supposed to be known.
\begin{proposition}
	Let $c=(a,m)^\mathsf{T}\in\R_{++}\times\R^n$ a known parameter vector, $b\in\R$, $\kappa\in\R^n$ and $\theta\in\mathcal{M}_n$. Let $(Z_t)_{t\in\R_+}$ be the unique strong solution of the SDE $\eqref{model0 Z}$ with initial random values $Z_0=(Y_0,X_0)^\mathsf{T}$ independent of $(B_t)_{t\in\R_{+}}$ satisfying $\mathbb{P}(Y_0\in\R_{++})=1$. Then, for each $T\in\R_{++}$ there exists a unique MLE of $\tau$ almost surely having the form
\begin{equation*}
	\tilde{\tau}_T^{\star}=\left( \displaystyle\int_0^T \tilde{\Lambda}(Z_s)^\mathsf{T} S(Z_s)^{-1} \tilde{\Lambda}(Z_s)\mathrm{~d}s\right)^{-1}\left(- \displaystyle\int_0^T \tilde{\Lambda}(Z_s)^\mathsf{T} S(Z_s)^{-1} \mathrm{d}Z_s +\displaystyle\int_0^T \tilde{\Lambda}(Z_s)^\mathsf{T} S(Z_s)^{-1} {c}\mathrm{~d}s\right),
\end{equation*}
where, for all $s\in[0,T]$, $S(Z_s)=Z_s^1 \rho\rho^{\mathsf{T}}$ and
$\tilde{\Lambda}(Z_s)=\left(\begin{matrix}
		Z_s^1& \mathbf{0}_{n(n+1)}^{\mathsf{T}} \\
		\mathbf{0}_{n}& I_n\otimes Z_s^{\mathsf{T}}\\\end{matrix}\right)$.	
	\label{MLE estimator supercritical}
\end{proposition}
\begin{proof}
Let $(\tilde{Z}_t)_{t\in\R_+}$ solution to $\mathrm{d}\tilde{Z}_t=\sqrt{\tilde{Z}_t^1}\,\rho\mathrm{~d}B_t$, for all $t\in\R_{+}$, with initial random value $\tilde{Z}_0$ satisfying $\mathbb{P}(\tilde{Z}_0=Z_0)=1$. Hence,
using the relation (7.138) in Lipster and Shiryaev \cite[page 297]{Lipster} applied on the processes $(Z_t)_{t\in[0,T]}$ and $(\tilde{Z}_t)_{t\in[0,T]}$, we deduce the associated log-likelihood ratio given by

\begin{equation*}
\ln L\left(\tilde{\tau},(Z_t)_{t\in[0,T]}\right)=\displaystyle\int_0^T ({c}-\tilde{\Lambda}(Z_s)\tilde{\tau})^\mathsf{T} S(Z_s)^{-1} \mathrm{~d}Z_s -\dfrac{1}{2} \displaystyle\int_0^T ({c}-\tilde{\Lambda}(Z_s)\tilde{\tau})^\mathsf{T} S(Z_s)^{-1} ({c}-\tilde{\Lambda}(Z_s)\tilde{\tau})\mathrm{~d}s.
\end{equation*}
with gradient vector
$$\triangledown_{\tilde{\tau}}\ln L(\tilde{\tau},(Z_t)_{t\in[0,T]})= - \displaystyle\int_0^T \tilde{\Lambda}(Z_s)^\mathsf{T} S(Z_s)^{-1} \mathrm{~d}Z_s +\displaystyle\int_0^T \tilde{\Lambda}(Z_s)^\mathsf{T} S(Z_s)^{-1} ({c}-\tilde{\Lambda}(Z_s)\tilde{\tau})\mathrm{~d}s.$$
Since the associated Hessian matrix $H_{\tau}L(\tilde{\tau})=-\displaystyle\int_0^T \tilde{\Lambda}(Z_s)^\mathsf{T} S(Z_s)^{-1} \tilde{\Lambda}(Z_s)\mathrm{~d}s$ is negative definite,  we deduce that the MLE $\tilde{\tau}_T^{\star}$ is the unique solution of $\triangledown_{\tilde{\tau}}L(\tilde{\tau})=\mathbf{0}_d$. this completes the proof. 
\end{proof}
\begin{remarque}
Similarly in spirit to Proposition \eqref{consistency Theorem} and to Theorem \eqref{Theorem asymptotic behavior subcritical case}, the asymptotic properties of the MLE $\tilde{\tau}_T^{\star}$ related to the drift parameter vector $\tilde{\tau}$ are can be proved in the subcritical case.
\end{remarque}
In order to study the asymptotic behavior of the MLE $\tilde{\tau}_T^{\star}$, we need the following integral version of Kronecker Lemma, see \cite[Lemma B.3.2]{kuchler}:
\begin{lemme}
	Let $g:\R_+ \to \R_+$ be a measurable function and for all $T\in\R_+$, $G(T)=\displaystyle\int_0^T g(t)\mathrm{~d}t$. If $\underset{T\to\infty}{\lim}G(T)=\infty$, then for every bounded and measurable function $f:\R_+\to\R$ for
	which the limit $f(\infty):=\underset{t\to\infty}{\lim}f(t)$ exists, we have
	\begin{equation*}
		\underset{T\to\infty}{\lim}\dfrac{1}{G(T)}\displaystyle\int_0^Tg(t)f(t)\mathrm{~d}t=f(\infty).
	\end{equation*}
	\label{Toeplitz type}
\end{lemme}

The next theorem states asymptotic behavior of the MLE $\tilde{\tau}_T^{\star}$ in a special supercritical case.
\begin{theoreme}
		Let $c=(a,m)^{\mathsf{T}}\in\R_{++}\times\R^n$ a known parameter vector, $\kappa\in\R^n$, $\theta\in\mathcal{M}_n$ a diagonalizable negative definite matrix and $b\in(\lambda_{\max}(\theta),0)$. Let $(Z_t)_{t\in\R_+}$ be the unique strong solution of the SDE $\eqref{model0 Z}$ with initial random values $Z_0=(Y_0,X_0)^\mathsf{T}$ independent of $(B_t)_{t\in\R_{+}}$ satisfying $\mathbb{P}(Y_0\in\R_{++})=1$. Let $P$ a modal matrix transforming $\theta$ to the diagonal matrix $D$. Suppose that $\text{diag}(P^{-1}m)P^{-1}\kappa\in\R_-^n$. Then
			\begin{equation}
		Q_T(\tilde{\tau}_T^{\star}-\tilde{\tau})\stackrel{\mathcal{D}}{\longrightarrow}(\eta\eta^{\mathsf{T}})^{-1}\eta\xi,\quad \text{as } T\to \infty,
		\label{distribution convergence special supercritical case}
	\end{equation}
	where $Q_T=\begin{bmatrix}
		e^{-\frac{b}{2}T}& \mathbf{0}_{n(n+1)}^{\mathsf{T}}\\
		\mathbf{0}_{n(n+1)}&I_n\otimes\tilde{Q}_T
	\end{bmatrix}$ with $\tilde{Q}_T=\text{diag}\left( e^{-\frac{b}{2}T},e^{\frac{(b-2\lambda_{\min}(\theta))}{2}T},\ldots,e^{\frac{(b-2\lambda_{\min}(\theta))}{2}T}\right) $, $\eta$ is defined through the almost sure limit 
	$	Q_T^{-1}\langle  \tilde{M} \rangle _T Q_T^{-1}\stackrel{a.s.}{\longrightarrow} \eta\eta^{\mathsf{T}},$ as $T\to \infty$, with $\langle  \tilde{M} \rangle _T=\displaystyle\int_0^T \tilde{\Lambda}(Z_s)^{\mathsf{T}} S(Z_s)^{-1} \tilde{\Lambda}(Z_s)\mathrm{~d}s$ and $\xi$ is a $d^2-n$-dimensional standard normally distributed random vector independent of $\eta$.
	\label{normality theorem supercritical case}
\end{theoreme}
\begin{proof}
	By writing the SDE associated to the process $Z$ as follows
	\begin{equation*}
			\mathrm{d}Z_t=(c-\tilde{\Lambda}(Z_t)\tilde{\tau})\mathrm{~d}t+\sqrt{Z_t^1}\rho \mathrm{~d}B_t,
	\end{equation*}
 it easy to check that, for $T\in\R_{++}$, we have
	\begin{multline*}
		\tilde{\tau}=\left( \displaystyle\int_0^T \tilde{\Lambda}(Z_s)^{\mathsf{T}} S(Z_s)^{-1} \tilde{\Lambda}(Z_s)\mathrm{~d}s\right)^{-1}\left(- \displaystyle\int_0^T \tilde{\Lambda}(Z_s)^{\mathsf{T}} S(Z_s)^{-1} \mathrm{~d}Z_s +\displaystyle\int_0^T \tilde{\Lambda}(Z_s)^{\mathsf{T}} S(Z_s)^{-1} {c}\mathrm{~d}s\right.\\+\left.\displaystyle\int_0^T\tilde{\Lambda}(Z_s)^{\mathsf{T}}S(Z_s)^{-1}\sqrt{Z_s^1}\,{\rho}\mathrm{~d}B_s\right)
	\end{multline*}  
	and hence the error term is written as $\tilde{\tau}_T^{\star}-\tilde{\tau}=\langle {\tilde{M}} \rangle_T^{-1} {\tilde{M}}_T,$ with ${\tilde{M}}_T=\displaystyle\int_0^T\frac{1}{\sqrt{Z_s^1}}\tilde{\Lambda}(Z_s)^{\mathsf{T}}({\rho}^{-1})^{\mathsf{T}}\mathrm{~d}B_s$ and $\langle {\tilde{M}} \rangle_T=\displaystyle\int_0^T \tilde{\Lambda}(Z_s)^{\mathsf{T}} S(Z_s)^{-1} \tilde{\Lambda}(Z_s)\mathrm{~d}s$. Moreover, the components of $\tilde{M}_T$ can be written explicitly, for all $k\in\lbrace1,\ldots,d\rbrace$ and $j\in\lbrace1,\ldots,n\rbrace$, as follows $${\tilde{M}}_T^{1}=\displaystyle\int_0^T \sqrt{Z_s^1}\displaystyle\sum_{i=1}^d ({\rho}^{-1})_{i,1} \mathrm{~d}B_s^j\quad\text{and}\quad {\tilde{M}}_T^{1+k+(j-1)d}=\displaystyle\int_0^T \dfrac{Z_s^k}{\sqrt{Z_s^1}}\displaystyle\sum_{i=1}^d ({\rho}^{-1})_{i,j} \mathrm{~d}B_s^j.$$
	Consequently, the components of $\langle {\tilde{M}} \rangle_T$ are taking, to within a constant, the integral form
$\displaystyle\int_0^T \dfrac{Z_s^i Z_s^j}{Z_s^1}  \mathrm{~d}s,$ where $i,j\in\lbrace1,\ldots,d\rbrace$. Next, since $b\in(\lambda_{\max}(\theta),0)$ and $\text{diag}(P^{-1}m)P^{-1}\kappa\in\mathbb{R}^n_-$, then  thanks to Lemma 4.4 and Lemma 7.2 of \cite{Boylog}, we obtain respectively the two following results:
	\textbf{first result}: there exists a random variable $G_1$ such that $e^{bt}Z_t^1 \stackrel{a.s.}{\longrightarrow}G_1$, as $t\to\infty$ and \textbf{second result}: for all $i\in\lbrace1,\ldots,n\rbrace$, there exists a random variable $\tilde{G}_i$ such that $e^{\lambda_i(\theta) t}\tilde{X}_t^i \stackrel{a.s.}{\longrightarrow}\tilde{G}_i$, as $t\to\infty$, where $(\tilde{X}_t)_{t\in\R_{+}}$ is the unique strong solution of the following SDE
	\begin{equation*}
				\mathrm{d}\tilde{X}_t=(\tilde{m}-\tilde{\kappa} Y_t-D \tilde{X}_t)\mathrm{~d}t+\sqrt{Y_t}\check{\rho}\mathrm{~d}B_t,	\end{equation*}
with initial value $\tilde{X}_0=P^{-1}X_0$, where $\tilde{m}=P^{-1}m$, $\tilde{\kappa}=P^{-1}\kappa$ and $\check{\rho}=P^{-1}\tilde{\rho}$.\\
Now, using the fact that $X_t=P \tilde{X}_t,$ for all $t\in\R_+$, we get for all $i\in\lbrace1,\ldots,n\rbrace$, $X_t^i=\displaystyle\sum_{k=1}^{n}P_{i,k}(\theta)\tilde{X}_t^k$.
	Since $e^{\lambda_k(\theta)}\tilde{X}_t^i\stackrel{a.s.}{\longrightarrow}\tilde{G}_i$, as $t\to\infty$, for all $i\in\lbrace 1,\ldots,n\rbrace$, we obtain
	\begin{equation*}
		e^{\lambda_{\min}(\theta)t} X_t^i=\displaystyle\sum_{k=1}^{n}P_{i,k}(\theta)e^{\lambda_{\min}(\theta)t}\tilde{X}_t^k \stackrel{a.s.}{\longrightarrow} P_{i,i_0}(\theta)\tilde{G}_{i_0}=: G_i(\theta),
	\end{equation*} 
	where $i_0\in\lbrace 1,\ldots,n\rbrace$ is the index associated to the smallest eigenvalue $\lambda_{\min}(\theta)$. 
	Hence, thanks to Lemma \ref{Toeplitz type}, for all $i,j\in\lbrace1,\ldots,n\rbrace$, we obtain
	\begin{equation*}
		\begin{aligned}
			e^{bT}\displaystyle\int_0^T Z_s^1 \mathrm{~d}s	&=-\dfrac{1}{b}(1-e^{bT})\dfrac{\displaystyle\int_0^T  e^{-bs}e^{bs}Z_s^1 \mathrm{~d}s}{\displaystyle\int_{0}^T e^{-bs}\mathrm{~d}s}\stackrel{a.s.}{\longrightarrow} -\dfrac{G_1}{b},\quad\text{ as}\;\; T\to\infty,
		\end{aligned}
	\end{equation*}	\begin{equation*}
		\begin{aligned}
			e^{\lambda_{\min}(\theta) T}\displaystyle\int_0^T Z_s^i \mathrm{~d}s&=-\dfrac{1}{\lambda_{\min}(\theta)}(1-e^{\lambda_{\min}(\theta) T})\dfrac{\displaystyle\int_0^T  e^{-\lambda_{\min}(\theta) s}e^{\lambda_{\min}(\theta) s}Z_s^i \mathrm{~d}s}{\displaystyle\int_{0}^T e^{-\lambda_{\min}(\theta) s}\mathrm{~d}s}\stackrel{a.s.}{\longrightarrow} \dfrac{-G_i(\theta)}{\lambda_{\min}(\theta)},\quad\text{ as}\;\; T\to\infty,
		\end{aligned}
	\end{equation*}
and
	\begin{equation*}
		\begin{aligned}
			e^{-(b-2\lambda_{\min}(\theta)) T}\displaystyle\int_0^T \dfrac{Z_s^i Z_s^j}{Z_s^1} \mathrm{~d}s	&=\dfrac{1}{b-2\lambda_{\min}(\theta)}\left( 1-e^{-(b-2\lambda_{\min}(\theta)) T}\right) \dfrac{\displaystyle\int_0^T  e^{(b-2\lambda_{\min}(\theta))s}\dfrac{e^{\lambda_{\min}(\theta) s}Z_s^i\,e^{\lambda_{\min}(\theta) s}Z_s^j}{e^{b s}Z_s^1}\mathrm{~d}s}{\displaystyle\int_{0}^T e^{(b-2\lambda_{\min}(\theta)) s}\mathrm{~d}s}\\
			&\stackrel{a.s.}{\longrightarrow} \dfrac{G_i(\theta) G_j(\theta)}{(b-2\lambda_{\min}(\theta))G_1},\quad\text{ as}\;\; T\to\infty,
		\end{aligned}
	\end{equation*}
	Consequently, the expression of $Q_T$ guaranteeing the almost sure convergence of $Q_T^{-1}\langle \tilde{M} \rangle _T Q_T^{-1}$ is deduced and the almost sure limit matrix $\eta\eta^{\mathsf{T}}$ of $Q_T^{-1}\langle \tilde{M} \rangle _T Q_T^{-1}$ does exist. Moreover, By Theorem \ref{CLT Van Zanten}, we get 
	$
		Q_T^{-1}  \tilde{M}_T\stackrel{\mathcal{D}}{\longrightarrow}\eta\xi,\text{ as } T\to \infty,
$
	where $\xi$ is a $d^2-n$-dimensional standard normally distributed random vector independent of $\eta$. Then, by Slutsky's Lemma, we have
	\begin{equation*}
		(Q_T^{-1} \tilde{M}_T,Q_T^{-1}\langle \tilde{M} \rangle _T Q_T^{-1})\stackrel{\mathcal{D}}{\longrightarrow}(\eta\xi,\eta\eta^{\mathsf{T}}),\quad \text{as } T\to \infty.
	\end{equation*}
	Finally, we obtain 
$Q_T(\tilde{\tau}_T^{\star}-\tilde{\tau})=(Q_T^{-1} \langle \tilde{M} \rangle_T Q_T^{-1})^{-1} (Q_T^{-1} \tilde{M}_T)\stackrel{\mathcal{D}}{\longrightarrow}(\eta\eta^{\mathsf{T}})^{-1}\eta\xi,\text{ as } T\to \infty.$
\end{proof}

\appendix
\begin{center}
\bf{\Large Appendix}
\end{center}

\newtheorem{Theoreme}{Theorem}
\newtheorem{Remarque}{Remark}
\setcounter{Theoreme}{0}

In what follows we recall some limit theorems for continuous local martingales used in the study of the asymptotic behavior of the MLE of $\tau$  First we recall a strong law of large numbers for continuous local martingales, see Liptser and Shiryaev \cite{Lipster}.

\begin{Theoreme}(Liptser and Shiryaev (2001)) Let $(\Omega,\mathcal{F},(\mathcal{F})_{t\in \real_{+}},\mathbb{P})$ be a filtered probability space satisfying the usual conditions. Let  $(M_{t})_{t \in \real_{+}}$ be a square-integrable continuous local martingale with respect to the filtration $(\mathcal{F}_t)_{t\in \real_{+}}$ such that $\mathbb{P} \left( M_0 = 0\right)=1.$ Let $(\xi_t)_{t\in \real_{+}}$ be a progressively measurable process such that
$$\mathbb{P}\left( \displaystyle\int_{0}^{t}  \xi_{u}^{2}\mathrm{~d} \left\langle M\right\rangle_u < \infty\right)=1, \quad t \in \real_{+},$$
and
\begin{align*}
\displaystyle\int_{0}^{t}  \xi_{u}^{2} \mathrm{~d} \left\langle M\right\rangle_u \stackrel{a.s.}{\longrightarrow} \infty,\quad as\;\; t\longrightarrow \infty,
\end{align*}
where $\left( \left\langle M\right\rangle_t  \right)_{t \in \real_{+}} $ denotes the quadratic variation process of $M.$ Then
\begin{align*}
\dfrac{\displaystyle\int_{0}^{t}  \xi_{u} \mathrm{~d} M_u}{\displaystyle\int_{0}^{t}  \xi_{u}^{2} \mathrm{~d} \left\langle M\right\rangle_u}
\stackrel{a.s.}{\longrightarrow} 0,\quad as\;\; t\longrightarrow \infty.
\end{align*}
If $(M_{t})_{t \in \real_{+}}$ is a standard Wiener process, the progressive measurability of  $(\xi_t)_{t\in \real_{+}}$ can be relaxed to measurability and adaptedness to the filtration $(\mathcal{F})_{t\in \real_{+}}.$
\label{Lipster Shiryaev theorem}
\end{Theoreme}
The second theorem is about the asymptotic behaviour of continuous multivariate local martingales, see \cite[Theorem 4.1]{Zanten}.
\begin{Theoreme}
	For $p\in\N \setminus \lbrace 0\rbrace$, let $M=(M_t)_{t\in\R_+}$ be a $p$-dimensional square-integrable continuous
	local martingale with respect to the filtration $(\mathcal{F}_t)_{t\in\R_+}$ such that $\mathbb{P}(M_0=0)=1$. Suppose that
	there exists a function $Q:\left[t_0,\infty\right) \to \mathcal{M}_p$ with some $t_0\in\R_+$ such that $Q(t)$ is an invertible
	(non-random) matrix for all $t\in\R_+$, $\underset{t\to\infty}{\lim} \norm{Q(t)}=0$ and
	$$	Q(t) \langle M\rangle_t Q(t)^{\mathsf{T}}\stackrel{\mathbb{P}}{\longrightarrow}\eta\eta^{\mathsf{T}},\quad \text{as }t\longrightarrow \infty,$$
	where $\eta$ is a random matrix in $\mathcal{M}_p$. Then, for each random matrix $A\in\mathcal{M}_{k,l}$,  $k,l\in\N\setminus\lbrace 0\rbrace$, defined on $(\Omega,\mathcal{F},(\mathcal{F})_{t\in \real_{+}},\mathbb{P})$, we have
	$$
	(Q(t)M_t,A) \stackrel{\mathcal{D}}{\longrightarrow}
	(\eta Z,A),\quad \text{as }t\longrightarrow \infty,$$
	where $Z$ is a $p$-dimensional standard normally distributed random vector independent of $(\eta,A)$.
	\label{CLT Van Zanten}
\end{Theoreme}

\bibliographystyle{plain}

\begin{thebibliography}{20}
\bibitem{Alfonsi} Alfonsi, A. (2015). Affine diffusions and related processes: simulation, theory and applications (Vol. 6, p. 20). Cham: Springer.
\bibitem{Barczy} Barczy, M., Ben Alaya, M. B., Kebaier, A., \& Pap, G. (2019). Asymptotic behavior of maximum likelihood estimators for a jump-type Heston model. Journal of Statistical Planning and Inference, 198, 139-164.
\bibitem{Barczy2} Barczy, M., Döring, L., Li, Z., \& Pap, G. (2014). Parameter estimation for a subcritical affine two factor model. Journal of Statistical Planning and Inference, 151, 37-59.
\bibitem{Barczy stationarity ergodicty} Barczy, M., Döring, L., Li, Z., \& Pap, G. (2014). Stationarity and ergodicity for an affine two-factor model. Advances in Applied Probability, 46(3), 878-898.
\bibitem{Basak}Basak, G. K., \& Lee, P. (2008). Asymptotic properties of an estimator of the drift coefficients of multidimensional Ornstein-Uhlenbeck processes that are not necessarily stable. Electronic Journal of Statistics, 2, 1309-1344.
\bibitem{Alaya} Ben Alaya, M., \& Kebaier, A. (2013). Asymptotic Behavior of The Maximum Likelihood Estimator For Ergodic and Nonergodic Sqaure-Root Diffusions. Stochastic Analysis and Applications, Taylor and Francis: STM, Behavioural Science and Public Health Titles.
\bibitem{Alaya1}Ben Alaya, M., \& Kebaier, A. (2012). Parameter estimation for the square-root diffusions: ergodic and nonergodic cases. Stochastic Models, 28(4), 609-634.
\bibitem{Bhattacharaya} Bhattacharya, R. N. (1982). On the functional central limit theorem and the law of the iterated logarithm for Markov processes. Zeitschrift für Wahrscheinlichkeitstheorie und verwandte Gebiete, 60(2), 185-201.
\bibitem{Boylog} Bolyog, B., \& Pap, G. (2019). On conditional least squares estimation for affine diffusions based on continuous time observations. Statistical Inference for Stochastic Processes. 22. 10.1007/s11203-018-9174-z. 
\bibitem{Boylog1} Bolyog, B. \& Pap, G. (2016). Conditions for Stationarity and Ergodicity of Two-factor Affine Diffusions. Communications on Stochastic Analysis. 10. 10.31390/cosa.10.4.12. 
\bibitem{Chung} Chung, K. L. (1982). Lectures from Markov Processes to Brownian Motion. Springer, New York.
\bibitem{Cox} Cox, J.C., Ingersoll, J.E., \& Ross, S.A. (1985). A theory of the term structure of interest rates. Econometrica 53(2):385–407.
\bibitem{Dai} Dai, Q., \& Singleton, K. J. (2000). Specification analysis of affine term structure models. The journal of finance, 55(5), 1943-1978.
\bibitem{Duffie} Duffie, D., Filipovic, D. \& Schachermayer, W. (2003). Affine processes and applications in finance. Ann. Appl. Prob. 13, 984–1053.
\bibitem{Fathallah}   Fathallah, H. (2013). Asymptotic properties of the LS-estimator of a Gaussian autoregressive process by an averaging method. Comm. Statist. Theory Methods 42, no. 17, 3148–3173.
\bibitem{Fathallah-Kebaier}   Fathallah, H. \& Kebaier, A. (2012).  Weighted limit theorems for continuous-time vector martingales with explosive and mixed growth. Stoch. Anal. Appl. 30, no. 2, 238–257.
\bibitem{Filipovic} Filipovic, D. \& Mayerhofer, E. (2009). Affine Diffusion Processes: Theory and Applications. SSRN Electronic Journal. 8.
\bibitem{Friesen}Friesen, M., Jin, P., Kremer, J., \& Rüdiger, B. (2020). Regularity of transition densities and ergodicity for affine jump-diffusion processes. arXiv preprint arXiv:2006.10009.
\bibitem{Friesen2} Friesen, M., Jin, P., \& Rüdiger, B. (2020). Existence of densities for multi-type continuous-state branching processes with immigration. Stochastic Processes and their Applications, 130(9), 5426-5452.
\bibitem{Getoor}Getoor, R. K. (1975). Markov Processes: Ray Processes and Right Processes. Lecture Notes in Mathematics.
Springer-Verlag Berlin Heidelberg.
\bibitem{Horn matrix}Horn, R. A., \& Johnson, C. R. (1994). Topics in matrix analysis. Cambridge university press.
\bibitem{Ikeda} Ikeda, N. \& Watanabe, S. (1981). Stochastic Differential Equations and Diffusion Processes,
North-Holland Publishing Company.
\bibitem{Jacod} Jacod, J., \& Shiryaev, A. (2013). Limit theorems for stochastic processes (Vol. 288). Springer Science and Business Media.
\bibitem{Jeanblanc} Jeanblanc, M., Yor, M., \& Chesney, M. (2009). Mathematical Methods for Financial Markets,
Springer-Verlag London Limited.
\bibitem{Jin} Jin, P., Kremer, J., \& Rüdiger, B. (2017). Exponential ergodicity of an affine two-factor model based on the $\alpha$-root process. Advances in Applied Probability, 49(4), 1144-1169.
\bibitem{Jin2} Jin, P., Kremer, J., \& Rüdiger, B. (2020). Existence of limiting distribution for affine processes. Journal of Mathematical Analysis and Applications, 486(2), 123912.
\bibitem{Jin 3} Jin, P., Rüdiger, B., \& Trabelsi, C. (2016). Exponential ergodicity of the jump-diffusion CIR process. In Stochastics of Environmental and Financial Economics: Centre of Advanced Study, Oslo, Norway, 2014-2015 (pp. 285-300). Springer International Publishing.
\bibitem{Karatzas} Karatzas, I., \& Shreve, S. (2012). Brownian motion and stochastic calculus (Vol. 113). Springer Science and Business Media.
\bibitem{Keller} Keller-Ressel, M., Schachermayer, W. \& Teichmann, J. (2011). Affine processes are regular. Probab. Theory Relat. Fields 151, 591–611.
\bibitem{Keller2013} Keller-Ressel, M., Schachermayer, W., \& Teichmann, J. (2013). Regularity of affine processes on general state spaces. Electronic journal of probability, 18, 1-17.
\bibitem{kuchler} Küchler, U. \& S\o rensen, M. (1997). Exponential families of stochastic processes, Springer-Verlag, New York.
\bibitem{Kutoyants2004} Kutoyants, Y. A. (2004). Statistical inference for ergodic diffusion processes. Springer Science and Business Media. 
\bibitem{Lipster} Liptser, R. S., \&  Shiryaev, A. N. (2001). Statistics of Random Processes II: II. Applications (Vol. 2). Springer Science and Business Media.
\bibitem{Meyn I}Meyn, S. P. \& Tweedie, R. L. (1992). Stability of Markovian processes. I. Criteria for discrete-time chains.
Adv. Appl. Prob. 24, 542–574.
\bibitem{Meyn II} Meyn, S. P. \& Tweedie, R. L. (1993). Stability of Markovian processes. II. Continuous-time processes and
sampled chains. Adv. Appl. Prob. 25, 487–517.
\bibitem{Foster-Lyapunov} Meyn, S. P. \& Tweedie, R. L. (1993). Stability of Markovian processes. III. Foster–Lyapunov criteria for continuous-time processes. Adv. Appl. Prob. 25, 518–548.
\bibitem{Meync'} Meyn, S. P. \& Tweedie, R. L. (2009). Markov Chain and Stochastic Stability, 2nd edition.
Cambridge University Press, Cambridge.
\bibitem{matrix cookbook}Petersen, K. B., \& Pedersen, M. S. (2012). The matrix cookbook (version: November 15, 2012).
\bibitem{Sharpe} Sharpe, M. (1988). General Theory of Markov Processes. Academic Press, Boston, MA.

\bibitem{Touati} Touati, A. (1992). On the functional convergence in distribution of sequences of semimartingales to a mixture of Brownian motions. Theory of Probability and Its Applications, 36(4), 752-771.
\bibitem{Van} Van der Vaart, A. W. (1998). Asymptotic Statistics, Cambridge University Press.
\bibitem{Zanten} Van Zanten, H. (2000). A multivariate central limit theorem for continuous local martingales. Statistics \& probability letters, 50(3), 229-235.
\bibitem{Zhang}Zhang, X., \& Glynn, P. W. (2018). Affine jump-diffusions: Stochastic stability and limit theorems. arXiv preprint arXiv:1811.00122.
\end{thebibliography}

\end{document}